\documentclass[preprint,12pt,times]{elsarticle}

\usepackage{amssymb}
\usepackage{amsmath}
\usepackage{amsthm}

\usepackage{geometry}
\usepackage{setspace}
\newtheorem{theorem}{Theorem}[section]

\newtheorem{lemma}[theorem]{Lemma}
\newtheorem{proposition}[theorem]{Proposition}
\theoremstyle{definition}

\theoremstyle{remark}
\newtheorem{remark}[theorem]{Remark}

\numberwithin{equation}{section}
\usepackage[colorlinks, linkcolor=blue, anchorcolor=blue, citecolor=blue, urlcolor=blue]{hyperref}
\usepackage{color}

\begin{document}

\begin{frontmatter}

\title{Weighted Ces\`aro type operators on the weighted Bergman spaces in the unit ball} %% Article title

\author{Jiaxin Pan\fnref{label1}}
\ead{jxpan0209@163.com}

\author{Cezhong Tong\fnref{label1}}
\ead{ctong@hebut.edu.cn; cezhongtong@hotmail.com}

\author{Zicong Yang\fnref{label1}} 
\ead{zc25@hebut.edu.cn; zicongyang@126.com}

\affiliation[label1]{Institute of Mathematics, Hebei University of Technology, Tianjin 300401, China}

    \begin{abstract}
In this paper, we study weighted Ces\`aro type operators
\[
\mathcal C_{\mu,\beta}^{\xi}:A_{\alpha}^{p}(\mathbb B_n)
\longrightarrow A_{\alpha}^{q}(\mathbb B_n)
\]
induced by measures on $[0,1)$, and obtain boundedness characterizations throughout the range
$0<p,q\le\infty$. We first extend the corresponding results on the unit disk to the unit ball in the range
$1\le p\le q<\infty$. By means of atomic decomposition, we further treat the case
$0<p<1$ and $p\le q<\infty$, thereby obtaining a complete characterization for
$0<p\le q<\infty$. The latter result is new even in one dimension. For the range $0<q<p\le\infty$, we characterize the boundedness in terms of an integral
condition involving the tail function of the inducing measure. This result is also new
even on the unit disk. Finally, we characterize the boundedness of the corresponding operators when the target
space is $H^\infty(\mathbb B_n)$.
\end{abstract}


\end{frontmatter}

\section{Introduction}

Let $\mathbb C^n$ be the $n$-dimensional complex Euclidean space. Denote by
\[\langle z,w\rangle=\sum\limits_{k=1}^n z_k\overline{ w}_k\quad \text{and}\quad |z|=\sqrt{\langle z,z\rangle}\]  for $z=(z_1,\dots,z_n)$ and $w=(w_1,\dots,w_n)$ in $\mathbb C^n$. 
Let $\mathbb B_n=\{z\in\mathbb C^n:|z|<1\}$ be the open unit ball in $\mathbb C^n$, and $\partial \mathbb B_n$ be its boundary. When $n=1$, write $\mathbb D:=\mathbb{B}_1$ for the unit disk in the complex plane $\mathbb{C}$. We use ${\rm d}v$ to represent the normalized volume measure on $\mathbb B_n$. For $\alpha>-1$, the normalized weighted volume measure is denoted by
\[
{\rm d}v_\alpha(z)=c_\alpha(1-|z|^2)^\alpha {\rm d} v(z),
\]
where $c_{\alpha}=\frac{\Gamma(n+1+\alpha)}{n!\Gamma(\alpha+1)}$ such that $v_\alpha(\mathbb B_n)=1$. 

Let $H(\mathbb B_n)$ be the
space of holomorphic functions on $\mathbb B_n$.
For $0<p<\infty$, denote by
$L_\alpha^p(\mathbb{B}_n)=L^p(\mathbb B_n,{\rm d}v_\alpha)$ the space of $p$-th integrable functions with the norm $\|\cdot\|_{p,\alpha}$, and by
$A_\alpha^p(\mathbb{B}_n)=L_\alpha^p(\mathbb{B}_n)\cap H(\mathbb B_n)$ the weighted Bergman spaces over the unit ball. We also write
$H^\infty(\mathbb B_n)$ for the space of bounded holomorphic functions on
$\mathbb B_n$, equipped with the supremum norm $\|\cdot\|_\infty$.

In this paper, we focus on the Ces\`aro operator, which is classical in complex analysis and functional analysis. The systematic study of Ces\`aro operators goes back to the work of Brown,
Halmos and Shields \cite{BHS1965}. Their boundedness and related
operator-theoretic properties on Hardy and Bergman spaces were later studied
by Andersen \cite{Andersen1996}, Siskakis
\cite{Siskakis1987,Siskakis1990,Siskakis1996}, Stempak
\cite{Stempak1994}, and others.
Closely related integration operators on Bergman spaces were studied by
Aleman and Siskakis \cite{AlemanSiskakis1997}.
In recent years, Ces\`aro type operators induced by measures have
attracted considerable attention, see for example
\cite{BaoSunWulan2022, GGM2022,GalGirMer2023}. In the unit disk, these
operators typically have the form
\[
    \mathcal C_{\mu,\beta} f(z)=\int_{[0,1)} \frac{f(tz)}{(1-tz)^\beta}\,{\rm d}\mu(t),
    \quad z\in\mathbb D,
\]
where $\mu$ is a positive Borel measure on $[0,1)$.

On the weighted Bergman space over the unit disk $A^p_\alpha(\mathbb D)$, Galanopoulos, Siskakis and Zhao
\cite{GSZ2025} characterized the boundedness of the weighted Ces\`aro type operator $\mathcal{C}_{\mu,\beta}$ from $A_{\alpha}^p(\mathbb{D})$ to $A_{\alpha}^q(\mathbb{D})$ for $1\leq p\leq q<\infty$ in terms of Carleson measures.
Related works on the unit disk have also been developed for Dirichlet and mixed norm spaces, see \cite{BlascoMas2025, GalMasMer2023,JinTang2022}.

In the several-variable setting, extended or generalized Ces\`aro type
operators and related integral operators on the unit ball were studied in
\cite{ChangLiStevic2007,Hu2003MixedNorm,Hu2004Bergman,LiStevic2009}. However, these works mainly concern classical or extended Ces\`aro operators rather than the measure-induced operators considered here.

Motivated by the results above, we introduce the weighted Ces\'aro type operator on the unit ball. Fix a boundary direction $\xi\in\partial\mathbb B_n$ and define
\[
    \mathcal C_{\mu,\beta}^{\xi}f(z)=\int_{[0,1)}\frac{f(tz)}{(1-t\langle z,\xi\rangle)^\beta}\,{\rm d}\mu(t),
    \quad z\in\mathbb B_n,
\] 
where $\mu $ is also a positive Borel measure on $[0,1)$.
It reduces to the operator on the disk when
$n=1$ and $\xi=1$.
The results below remove the restrictions $p,q\geq1$.  They cover the ranges
$0<p\leq q<\infty$ and $0<q<p\leq\infty$, together with the cases in which
the target space is $H^\infty(\mathbb B_n)$.

{\bf Notation}. In what follows, $\xi\in\partial\mathbb{B}_n$ is fixed and for $k\in \mathbb{N}$, denote by $r_k=1-2^{-k}$ and $\delta_k=2^{-k}$. Given a positive Borel measure $\mu$ on $[0,1)$, set $\hat{\mu}(r)=\mu([r,1))$ for $r\in [0,1)$.

 We also put $\delta_0=1$, $r_0=0$, and
\[
    I_k=[r_k,r_{k+1}),\qquad k\geq0,
\]
so that $I_0=[0,1/2)$.  As usual, $1/\infty=0$, and we write
$A_\alpha^\infty(\mathbb B_n)=H^\infty(\mathbb B_n)$ and
$\|\cdot\|_{\infty,\alpha}:=\|\cdot\|_\infty$.

 Now we are ready to introduce our main results. Our first result extends Theorem 2 in \cite{GSZ2025} to the ball setting  and to the range $0<p\leq q<\infty$.
\begin{theorem}\label{theorem1.1}
Let $\alpha>-1$, $\beta>0$ and   $0<p\leq q<\infty$. Suppose $\mu$ is a positive Borel measure on $[0,1)$. Then $\mathcal{C}_{\mu,\beta}^{\xi}: A_{\alpha}^p(\mathbb{B}_n)\to A_{\alpha}^q(\mathbb{B}_n)$ is bounded if and only if $\mu$ is a $\lambda$-Carleson measure on $[0,1)$, where \[
    \lambda=\beta+(n+1+\alpha)\Big(\frac{1}{p}-\frac{1}{q}\Big).
\] That is
\[
    \sup_{r\in [0,1)}\frac{\hat{\mu}(r)}{(1-r)^{\lambda}}<\infty.
\]
\end{theorem}

Our second result concerns the   range $0<q<p\leq\infty$ on the unit ball; the
case $1\leq q<p<\infty$  is completely new even in one
dimensional case.

\begin{theorem}\label{theorem1.2}
Let $\alpha>-1$, $\beta>0$ and   $0<q<p\leq\infty$. Suppose $\mu$ is a positive Borel measure on $[0,1)$.
Then $\mathcal{C}_{\mu,\beta}^{\xi}: A_{\alpha}^p(\mathbb{B}_n)\to A_{\alpha}^q(\mathbb{B}_n)$ is bounded if and only if $\mu$ is finite and 
\begin{equation}\label{equa1.1}
    \int_{0}^1\left(\frac{\hat{\mu}(r)}{(1-r)^{\lambda}}\right)^{\theta}\frac{{\rm d}r}{1-r}<\infty.
\end{equation}
Here
  \[
    \frac{1}{\theta}=\frac{1}q-\frac{1}p
    \quad \text{and}\quad
    \lambda=\beta+(n+1+\alpha)\Big(\frac{1}{p}-\frac{1}{q}\Big)
    =\beta-\frac{n+1+\alpha}{\theta}.
\]
For finite $\mu$, condition~\eqref{equa1.1} is equivalent to
\[
    \left\{\frac{\hat\mu(r_k)}{\delta_k^\lambda}\right\}_{k\geq1}
    \in\ell^\theta.
\]
 
\end{theorem}

We also characterize   boundedness into  $H^{\infty}(\mathbb{B}_n)$.

\begin{theorem}\label{theorem1.3}
Let $\alpha>-1$, $\beta>0$ and $\mu$ be a positive Borel measure on $[0,1)$. 
Then the following assertions hold.
\begin{itemize}
\item[(i)]   For $0<p\leq1$, the operator $\mathcal{C}_{\mu,\beta}^{\xi}: A_{\alpha}^p(\mathbb{B}_n)\to H^{\infty}(\mathbb{B}_n)$  is bounded if and only if
  \[
    \sup_{0\leq r<1}
    \frac{\hat\mu(r)}{(1-r)^{\beta+(n+1+\alpha)/p}}<\infty.
\] 
\item[(ii)] For   $1<p<\infty$, let $\frac{1}{p}+\frac{1}{p'}=1$.  Then $\mathcal{C}_{\mu,\beta}^{\xi}: A_{\alpha}^p(\mathbb{B}_n)\to H^{\infty}(\mathbb{B}_n)$ is bounded if and only if $\mu$ is finite and
\[
    \left\{\frac{\hat{\mu}(r_k)}{\delta_k^{\,\beta+(n+1+\alpha)/p}}\right\}_{k\geq 1}\in \ell^{p'}.
\]
\item[(iii)]   The operator $\mathcal{C}_{\mu,\beta}^{\xi}: H^{\infty}(\mathbb{B}_n)\to H^{\infty}(\mathbb{B}_n)$  is bounded if and only if
  \[
    \int_{[0,1)}\frac{{\rm d}\mu(t)}{(1-t)^{\beta}}<\infty.
\] 
\end{itemize}
\end{theorem}

This paper is organized as follows.
In Section 2, we recall several Forelli--Rudin type estimates, the generalized Schur's test,  atomic decomposition, a radial maximal estimate,  and other auxiliary results needed in the sequel.
Section 3 is devoted to the proof of Theorem~\ref{theorem1.1}, where the sufficiency and necessity are established separately.
In Section 4, we prove Theorem~\ref{theorem1.2} by deriving both integral and discrete characterizations involving the tail function $\hat{\mu}$.
Finally, in Section 5, we study the boundedness of $\mathcal{C}_{\mu,\beta}^{\xi}$ between weighted Bergman spaces and $H^{\infty}(\mathbb{B}_n)$, thereby proving Theorem~\ref{theorem1.3}.

Throughout the paper, $C>0$ denotes a constant that may change from line
to line but is independent of the relevant variables. For non-negative quantities $A$ and $B$, we write $A\lesssim B$ (or equivalently $B\gtrsim A$) if there exists an absolute constant $C>0$ such that $A\leq CB$, and $A\simeq B$ means that both $A\lesssim B$ and $B\lesssim A$. 

\section{Preliminaries}

In this section, we recall some basic facts and present some auxiliary lemmas that will be used in the sequel.

The following two lemmas are the Forelli-Rudin type estimates. For related
$L^p$--$L^q$ mapping results for Forelli--Rudin type operators on the unit
ball, see \cite{ZhaoZhou2022}.

\begin{lemma}[Theorem 1.12 in \cite{Zhu2005}]\label{lemma2.1}
Suppose $\delta>-1$ and $c$ is real. Let
\[
    J_{c,\delta}(z)=\int_{\mathbb{B}_n}\frac{(1-|w|^2)^{\delta}}{|1-\langle z,w\rangle|^{n+1+\delta+c}}{\rm d}v(w),\quad z\in\mathbb{B}_n.
\]
Then the following asymptotic properties hold.
\begin{itemize}
\item[(1)] If $c<0$, then $J_{c,\delta}$ is bounded in $\mathbb{B}_n$.
\item[(2)] If $c=0$, then
\[
    J_{0,\delta}(z)\simeq \log\frac{1}{1-|z|^2},\quad |z|\to 1^-.
\]
\item[(3)] If $c>0$, then
\[
    J_{c,\delta}(z)\simeq \frac{1}{(1-|z|^2)^c},\quad |z|\to 1^{-}.
\]
\end{itemize}
\end{lemma}

\begin{lemma}[Theorem 3.1 in \cite{ZhangLiShangGuo2018}]\label{lemma2.2}
Suppose $b\geq 0$, $c\geq 0$ and $\delta>-1$. Let 
\[
    J_{b,c,\delta}(z,w)=\int_{\mathbb{B}_n}\frac{(1-|u|^2)^{\delta}}{|1-\langle z,u\rangle|^{b}|1-\langle w,u\rangle|^c}{\rm d}v(u),\quad z,w\in\mathbb{B}_n.
\]
Then the following asymptotic properties hold.
\begin{itemize}
\item[(1)] If $b+c<n+1+\delta$, then $J_{b,c,\delta}$ is bounded.
\item[(2)] If $b+c>n+1+\delta$, $b<n+1+\delta$ and $c<n+1+\delta$, then 
\[
    J_{b,c,\delta}(z,w)\simeq \frac{1}{|1-\langle z,w\rangle|^{b+c-n-1-\delta}}.
\]
\item[(3)] If $b>n+1+\delta>c$, then 
\[
    J_{b,c,\delta}(z,w)\simeq \frac{1}{(1-|z|^2)^{b-n-1-\delta}|1-\langle z,w\rangle|^c}.
\]
\item[(4)] If $b>n+1+\delta$ and $c>n+1+\delta$, then
\begin{align*}
J_{b,c,\delta}(z,w)\simeq \frac{1}{(1-|z|^2)^{b-n-1-\delta}|1-\langle z,w\rangle|^c}+\frac{1}{(1-|w|^2)^{c-n-1-\delta}|1-\langle z,w\rangle|^{b}}.
\end{align*}
\end{itemize}
\end{lemma}

We next recall the generalized Schur's test. For related forms for nonnegative
integral operators, see \cite{Okikiolu1970}; for an
$L^p$--$L^q$ version on the unit ball, see \cite{Zhao2015}.

\begin{lemma}[Theorem A in \cite{GSZ2025}]\label{lemma2.3}
Let $\mu_1$ and $\mu_2$ be two positive measures on the space $X$ and let $K(x,y)$ be a nonnegative measurable function on $X\times X$. Let $T$ be the integral operator defined with kernel $K$ as follows:
\[
    Tf(x)=\int_{X}K(x,y)f(y){\rm d}\mu_1(y).
\]
Suppose $1<p\leq q<\infty$ and $\frac{1}{p}+\frac{1}{p'}=1$. Let $\gamma_1, \gamma_2>0$ with $\gamma_1+\gamma_2=1$. Assume that there exist two positive functions $h_1, h_2$ and positive constants $C_1, C_2$ such that
\[
    \int_{X}h_1(y)^{p'}K(x,y)^{\gamma_1 p'}{\rm d}\mu_1(y)\leq C_1h_2(x)^{p'}
\]
for almost all $x\in X$, and 
\[
    \int_{X}h_2(x)^qK(x,y)^{\gamma_2 q}{\rm d}\mu_2(x)\leq C_2h_1(y)^q
\]
for almost all $y\in X$. Then the operator $T: L^p(X,{\rm d}\mu_1)\to L^q(X, {\rm d}\mu_2)$ is bounded.
\end{lemma}

 We next recall the lattice of the atomic decomposition for weighted
Bergman spaces.

\begin{proposition}[Theorem 32 in \cite{ZhaoZhu2008}]\label{proposition2.4}
Let $\alpha>-1$, $0<p<\infty$, and choose
\[
    M>n\max\left\{1,\frac{1}p\right\}+\frac{\alpha+1}{p}.
\]
There is a Bergman-metric separated lattice
$\{w_j\}\subset\mathbb B_n$ such that every $f\in A_\alpha^p(\mathbb B_n)$
has a representation
\begin{equation}\label{atomic-decomposition}
    f(z)=\sum_{j=1}^{\infty}c_j a_{w_j,M}^{(p)}(z),
    \qquad
    a_{w,M}^{(p)}(z)=
    \frac{(1-|w|^2)^{M-(n+1+\alpha)/p}}{(1-\langle z,w\rangle)^M},
    \end{equation}
where $\{c_j\}\in\ell^p$.  The series converges in $A_\alpha^p$ and
locally uniformly, and
\[
    \|f\|_{p,\alpha}\simeq
    \inf\left\{\|\{c_j\}\|_{\ell^p}:
    f=\sum_jc_ja_{w_j,M}^{(p)}\right\}.
\]
\end{proposition}

We shall use the following ball version of the radial maximal estimate in \cite[Theorem~2.1]{AguilarHernandezMasPelaezRattya2026}.
\begin{lemma}\label{lemma2.5}
For $f\in H(\mathbb B_n)$, set
$\mathcal Rf(z)=\sup_{0\leq s\leq1}|f(sz)|$. Then
\begin{equation}\label{equa-radial-maximal}
    \|\mathcal Rf\|_{L_\alpha^p}\lesssim\|f\|_{p,\alpha},
    \qquad 0<p<\infty,\quad f\in A_\alpha^p(\mathbb B_n).
\end{equation}
\end{lemma}

\begin{proof}
Fix $\eta\in\partial\mathbb B_n$ and define $g_\eta(w)=f(w\eta)$ for $w\in\mathbb D$. Apply \cite[Theorem~2.1]{AguilarHernandezMasPelaezRattya2026} to $g_\eta$ with $M=1$. With $q=p$ and $\omega(r)=r^{2n-2}(1-r^2)^\alpha$, the corresponding mixed norm becomes
\[
\|h\|_{L_p^p(\omega)}^p
=\frac{1}{2\pi}\int_0^{2\pi}\int_0^1
|h(re^{i\theta})|^p r^{2n-1}(1-r^2)^\alpha
\,{\rm d}r\,{\rm d}\theta.
\]
Hence
\begin{equation}\label{equa-radial-maximal-slice}
\frac{1}{2\pi}\int_0^{2\pi}\int_0^1
R(g_\eta)(re^{i\theta})^p r^{2n-1}(1-r^2)^\alpha
\,{\rm d}r\,{\rm d}\theta
\lesssim
\frac{1}{2\pi}\int_0^{2\pi}\int_0^1
|g_\eta(re^{i\theta})|^p r^{2n-1}(1-r^2)^\alpha
\,{\rm d}r\,{\rm d}\theta.
\end{equation}
Moreover,
\[
R(g_\eta)(re^{i\theta})
=\sup_{0\le s\le1}|f(sre^{i\theta}\eta)|
=\mathcal Rf(re^{i\theta}\eta),
\qquad
g_\eta(re^{i\theta})=f(re^{i\theta}\eta).
\]
Integrating \eqref{equa-radial-maximal-slice} over $\eta\in\partial\mathbb B_n$ with respect to the normalized surface measure ${\rm d}\sigma$, and using the invariance of ${\rm d}\sigma$ under $\eta\mapsto e^{i\theta}\eta$, we obtain
\[
\begin{aligned}
  &\int_{\partial\mathbb B_n}\int_0^1\mathcal Rf(r\eta)^p
    r^{2n-1}(1-r^2)^\alpha\,{\rm d}r\,{\rm d}\sigma(\eta) \lesssim
    \int_{\partial\mathbb B_n}\int_0^1|f(r\eta)|^p
    r^{2n-1}(1-r^2)^\alpha\,{\rm d}r\,{\rm d}\sigma(\eta).
\end{aligned}
\]
The weighted polar-coordinate formula now gives
\eqref{equa-radial-maximal}.
\end{proof}

 Now we look into the measures on $[0,1)$.

\begin{lemma}[Proposition 2.1 in \cite{BaoSunWulan2022}]  \label{lemma2.6}
Let $\gamma>0$, $0\leq\rho<s$, and let $\mu$ be a positive Borel measure on $[0,1)$. Then $\mu$ is an $s$-Carleson measure on $[0,1)$, that is,
\[
    \sup_{r\in [0,1)}\frac{\hat{\mu}(r)}{(1-r)^{s}}<\infty,
\]
if and only if
\[
    \sup_{a\in [0,1)}\int_{[0,1)}\frac{(1-a)^{\gamma}}{(1-t)^\rho(1-at)^{s-\rho+\gamma}}{\rm d}\mu(t)<\infty.
\]
\end{lemma}

The next lemma is a key tool for studying the case $1<q<p<\infty$.
\begin{lemma}  \label{lemma2.7}
  Let $\lambda\in\mathbb{R}$, $\theta>0$ and $\mu$ be a finite positive Borel measure on $[0,1)$. Then
\begin{equation}\label{equa2.1}
\int_{0}^1\Big(\frac{\hat{\mu}(r)}{(1-r)^{\lambda}}\Big)^{\theta}\frac{{\rm d}r}{1-r}<\infty
\end{equation}
if and only if 
\begin{equation}\label{equa2.2}
\sum_{k=1}^{\infty}\Big(\frac{\hat{\mu}(r_k)}{\delta_k^{\lambda}}\Big)^{\theta}<\infty.
\end{equation}
\end{lemma}

\begin{proof}
For $k\geq 0$, recall that $\delta_k=2^{-k}$ and $r_k=1-2^{-k}$.
Denote by 
\[
    I_k=[1-\delta_k,1-\delta_{k+1}),\quad k\geq 0.
\]
Then $\int_{I_k}\frac{{\rm d}r}{1-r}=\log 2$. When $r\in I_k$, we have $1-r\simeq \delta_k=2\delta_{k+1}$ and 
\[
    \hat{\mu}(r_k)\geq \hat{\mu}(r)\geq \hat{\mu}(r_{k+1}).
\]
It follows that
\begin{equation*}
\left(\frac{\hat{\mu}(r_{k+1})}{\delta_{k+1}^{\lambda}}\right)^{\theta}\lesssim \int_{I_k}\left(\frac{\hat{\mu}(r)}{(1-r)^{\lambda}}\right)^{\theta}\frac{{\rm d}r}{1-r}\lesssim \left(\frac{\hat{\mu}(r_k)}{\delta_k^{\lambda}}\right)^{\theta}.
\end{equation*}
Therefore,
\begin{align*}
\sum_{k=1}^{\infty}\left(\frac{\hat{\mu}(r_k)}{\delta_k^{\lambda}}\right)^{\theta}&\lesssim \sum_{k=0}^{\infty}\int_{I_k}\left(\frac{\hat{\mu}(r)}{(1-r)^{\lambda}}\right)^{\theta}\frac{{\rm d}r}{1-r}=\int_{0}^1\left(\frac{\hat{\mu}(r)}{(1-r)^{\lambda}}\right)^{\theta}\frac{{\rm d}r}{1-r}\\
&\lesssim \int_{0}^{\frac{1}{2}}\left(\frac{\hat{\mu}(r)}{(1-r)^{\lambda}}\right)^{\theta}\frac{{\rm d}r}{1-r}+\sum_{k=1}^{\infty}\left(\frac{\hat{\mu}(r_k)}{\delta_k^{\lambda}}\right)^{\theta}.
\end{align*}
Note that the integral over $[0,\frac{1}{2}]$ is finite because $\mu$ is finite. Hence the equivalence between \eqref{equa2.1} and \eqref{equa2.2} is established.
\end{proof}

With the same strategy to study the operator $\mathcal C_{\mu,\beta}$ on unit disk, we write the operator $\mathcal{C}_{\mu,\beta}^{\xi}$ as an integral operator with a kernel.

\begin{lemma}\label{lemmaKR}
Let $\alpha>-1$, $\beta>0$ and $\delta\geq 0$. If $f\in H(\mathbb B_n)$, then the following formula holds:
\begin{equation}\label{equa2.3}
    \big(\mathcal{C}_{\mu,\beta}^{\xi}f\big)(z)=\frac{c_{\alpha+\delta}}{c_{\alpha}}\int_{\mathbb{B}_n}f(w)K_{\mu,\beta,\delta}^{\xi}(z,w){\rm d}v_{\alpha}(w),
\end{equation}
where 
\begin{equation}\label{equa2.4}
K_{\mu,\beta,\delta}^{\xi}(z,w)=\int_{[0,1)}\frac{(1-|w|^2)^{\delta}}{(1-t\langle z,\xi\rangle)^{\beta}(1-t\langle z,w\rangle)^{n+1+\alpha+\delta}}{\rm d}\mu(t).
\end{equation}
\end{lemma}

\begin{proof}
For any $t\in [0,1)$, the function $z\mapsto f(tz)$ is holomorphic on $\overline{\mathbb{B}_n}$. By the reproducing formula (see \cite[Theorem 2.2]{Zhu2005}), we get
\[
    f(tz)=\frac{c_{\alpha+\delta}}{c_{\alpha}}\int_{\mathbb{B}_n}\frac{f(w)(1-|w|^2)^{\delta}}{(1-t\langle z,w\rangle)^{n+1+\alpha+\delta}}{\rm d}v_{\alpha}(w).
\]
Then the formula \eqref{equa2.3} is established by Fubini's theorem.
\end{proof}

\section{Proof of Theorem 1.1}

 In this section, we prove Theorem \ref{theorem1.1}. To prove the sufficiency, we introduce an operator analogous to the one in \cite{GSZ2025}. Let 
\[
    \widetilde{K}_{\mu,\beta,\delta}^{\xi}(z,w)=\int_{[0,1)}\frac{(1-|w|^2)^{\delta}}{|1-t\langle z,\xi\rangle|^{\beta}|1-t\langle z,w\rangle|^{n+1+\alpha+\delta}}{\rm d}\mu(t),
\]
and define
\begin{equation}\label{equa3.01}
\mathcal{B}_{\mu,\beta,\delta}^{\xi}f(z)=\frac{c_{\alpha+\delta}}{c_{\alpha}}\int_{\mathbb{B}_n}f(w)\widetilde{K}_{\mu,\beta,\delta}^{\xi}(z,w){\rm d}v_{\alpha}(w).
\end{equation}
Clearly, if $\mathcal{B}_{\mu,\beta,\delta}^{\xi}: L_{\alpha}^p(\mathbb{B}_n)\to L_{\alpha}^q(\mathbb{B}_n)$ is bounded for some (any) $\delta\geq 0$ then $\mathcal{C}_{\mu,\beta}^\xi: A_{\alpha}^p(\mathbb{B}_n)\to A_{\alpha}^q(\mathbb{B}_n)$ is bounded.

We first separate the sufficiency into two cases: $p>1$ and $p=1$.
\begin{proposition}\label{proposition3.1}
Let $\alpha>-1$, $\beta>0$ and $1<p\leq q<\infty$. Set 
\[
    \lambda=\beta+(n+1+\alpha)\Big(\frac{1}{p}-\frac{1}{q}\Big).
\]
If $\mu$ is a $\lambda$-Carleson measure on $[0,1)$, then $\mathcal{B}_{\mu,\beta,\delta}^{\xi}: L_{\alpha}^p(\mathbb{B}_n)\to L_{\alpha}^q(\mathbb{B}_n)$ is bounded for every $\delta\geq 0$. Consequently, $\mathcal{C}_{\mu,\beta}^{\xi}: A_{\alpha}^p(\mathbb{B}_n)\to A_{\alpha}^q(\mathbb{B}_n)$ is bounded.
\end{proposition}

\begin{proof}
We are going to use the generalized Schur test in Lemma \ref{lemma2.3} to prove the boundedness of $\mathcal{B}_{\mu,\beta,\delta}^{\xi}$ from $L_{\alpha}^p(\mathbb{B}_n)$ to $L_{\alpha}^q(\mathbb{B}_n)$ for any $\delta\geq 0$. To this end, let $p'=p/(p-1)$ and set 
\begin{equation*}
\gamma_1=\frac{q}{p'+q},\quad \gamma_2=\frac{p'}{p'+q}.
\end{equation*}
Clearly, $\gamma_1+\gamma_2=1$ and
\begin{equation}\label{equa3.1}
p'\gamma_1=q\gamma_2=\frac{1}{1-1/p+1/q}\geq 1.
\end{equation}
Choose a number $s$ such that 
\begin{equation}\label{equa3.2}
    0<s<\min\left\{\frac{\alpha+1}{q},\frac{\alpha+1}{p'}\right\}.
\end{equation}
Let 
\[
    h_s(z)=(1-|z|^2)^{-s},\quad z\in\mathbb B_n.
\]
 Since $\delta\geq0$ and
\[
    1-|w|^2\leq2|1-t\langle z,w\rangle|,
\]
the factor $(1-|w|^2)^\delta/|1-t\langle z,w\rangle|^\delta$ is bounded.  Using Minkowski's inequality, we obtain
\begin{align*}
&\quad \int_{\mathbb{B}_n}h_s(w)^{p'}\widetilde{K}_{\mu,\beta,\delta}^{\xi}(z,w)^{\gamma_1p'}{\rm d}v_{\alpha}(w)\\
&\lesssim \int_{\mathbb{B}_n}(1-|w|^2)^{\alpha-sp'}\left(\int_{[0,1)}\frac{{\rm d}\mu(t)}{|1-t\langle z,\xi\rangle|^{\beta}|1-t\langle z,w\rangle|^{n+1+\alpha}}\right)^{\gamma_1p'}{\rm d}v(w)\\
&\lesssim \left[\int_{[0,1)}\frac{1}{|1-t\langle z,\xi\rangle|^{\beta}}\left(\int_{\mathbb{B}_n}\frac{(1-|w|^2)^{\alpha-sp'}}{|1-t\langle z,w\rangle|^{(n+1+\alpha)\gamma_1p'}}{\rm d}v(w)\right)^{1/(\gamma_1p')}{\rm d}\mu(t)\right]^{\gamma_1p'}.
\end{align*}
By \eqref{equa3.1} and \eqref{equa3.2}, we have $\alpha-sp'>-1$ and 
\[
    (n+1+\alpha)\gamma_1p'-(n+1)-(\alpha-sp')=(n+1+\alpha)(\gamma_1p'-1)+sp'>0.
\]
So it follows from Lemma \ref{lemma2.1} that
\begin{align*}
\int_{\mathbb{B}_n}\frac{(1-|w|^2)^{\alpha-sp'}}{|1-t\langle z,w\rangle|^{(n+1+\alpha)\gamma_1p'}}{\rm d}v(w)\lesssim \frac{1}{(1-t|z|)^{(n+1+\alpha)(\gamma_1p'-1)+sp'}}.
\end{align*}
Consequently, 
\begin{align*}
\int_{\mathbb{B}_n}h_s(w)^{p'}\widetilde{K}_{\mu,\beta,\delta}^{\xi}(z,w)^{\gamma_1p'}{\rm d}v_{\alpha}(w)\lesssim \left(\int_{[0,1)}\frac{{\rm d}\mu(t)}{(1-t|z|)^{\beta+(n+1+\alpha)(1-1/(\gamma_1p'))+s/\gamma_1}}\right)^{\gamma_1p'}.
\end{align*}
Note that $\beta+(n+1+\alpha)(1-\frac{1}{\gamma_1p'})+\frac{s}{\gamma_1}=\beta+(n+1+\alpha)(\frac{1}{p}-\frac{1}{q})+\frac{s}{\gamma_1}=\lambda+\frac{s}{\gamma_1}$ and $\frac{s}{\gamma_1}>0$. Using Lemma \ref{lemma2.6} with $\rho=0$, if $\mu$ is a $\lambda$-Carleson measure on $[0,1)$, then 
  \begin{align*}
\int_{[0,1)}\frac{{\rm d}\mu(t)}{(1-t|z|)^{\beta+(n+1+\alpha)(1-1/(\gamma_1p'))+s/\gamma_1}}\lesssim \frac{1}{(1-|z|^2)^{s/\gamma_1}},\quad \forall z\in\mathbb{B}_n.
\end{align*} 
Therefore,
\begin{equation}\label{equa3.3}
\int_{\mathbb{B}_n}h_s(w)^{p'}\widetilde{K}_{\mu,\beta,\delta}^{\xi}(z,w)^{\gamma_1p'}{\rm d}v_{\alpha}(w)\lesssim \left(\frac{1}{(1-|z|^2)^{s/\gamma_1}}\right)^{\gamma_1p'}=h_s(z)^{p'}.
\end{equation}

On the other hand, since $\gamma_2q\geq 1$, we again use Minkowski's inequality to obtain
\begin{equation}\label{equa3.4}
\begin{split}
&\quad \int_{\mathbb{B}_n}h_s(z)^q\widetilde{K}_{\mu,\beta,\delta}^{\xi}(z,w)^{\gamma_2q}{\rm d}v_{\alpha}(z)\\
&\lesssim \int_{\mathbb{B}_n}(1-|z|^2)^{\alpha-sq}\left(\int_{[0,1)}\frac{{\rm d}\mu(t)}{|1-t\langle z,\xi\rangle|^{\beta}|1-t\langle z,w\rangle|^{n+1+\alpha}}\right)^{\gamma_2q}{\rm d}v(z)\\
&\lesssim \left[\int_{[0,1)}\left(\int_{\mathbb{B}_n}\frac{(1-|z|^2)^{\alpha-sq}{\rm d}v(z)}{|1-t\langle z,\xi\rangle|^{\beta\gamma_2q}|1-t\langle z,w\rangle|^{(n+1+\alpha)\gamma_2q}}\right)^{1/(\gamma_2q)}{\rm d}\mu(t)\right]^{\gamma_2q}.
\end{split}
\end{equation}
Let $\delta'=\alpha-sq$, $b'=\beta\gamma_2q$ and $c'=(n+1+\alpha)\gamma_2q$. Then by \eqref{equa3.1} and \eqref{equa3.2}, we know that   $\delta'>-1$  and
\[
    c'-(n+1+\delta')=(n+1+\alpha)(\gamma_2q-1)+sq>0.
\]
Next we will discuss three cases: $b'<n+1+\delta'$, $b'>n+1+\delta'$ and $b'=n+1+\delta'$.

$\bullet${\it Case 1. $b'<n+1+\delta'$}. In this case, by item (3) in Lemma \ref{lemma2.2}, we obtain
\begin{equation}\label{equa3.5}
\begin{split}
&\quad \int_{\mathbb{B}_n}\frac{(1-|z|^2)^{\alpha-sq}}{|1-t\langle z,\xi\rangle|^{\beta\gamma_2 q}|1-t\langle z,w\rangle|^{(n+1+\alpha)\gamma_2 q}}{\rm d}v(z)\\
&\simeq \frac{1}{(1-t^2|w|^2)^{c'-(n+1+\delta')}}\cdot\frac{1}{|1-t^2\langle w,\xi\rangle|^{b'}}\\
&\lesssim \frac{1}{(1-t|w|)^{(n+1+\alpha)(\gamma_2 q-1)+sq+\beta\gamma_2 q}}=\frac{1}{(1-t|w|)^{\lambda\gamma_2 q+sq}}.
\end{split}
\end{equation}
Since $\mu$ is a $\lambda$-Carleson measure on $[0,1)$, by \eqref{equa3.4}, \eqref{equa3.5} and Lemma \ref{lemma2.6} (with $\rho=0$), we obtain
\begin{equation}\label{equa3.6}
\begin{split}
\int_{\mathbb{B}_n}h_s(z)^q\left[\widetilde{K}_{\mu,\beta,\delta}^{\xi}(z,w)\right]^{\gamma_2q}{\rm d}v_{\alpha}(z)&\lesssim \left(\int_{[0,1)}\frac{{\rm d}\mu(t)}{(1-t|w|)^{\lambda+s/\gamma_2}}\right)^{\gamma_2 q}\\
&\lesssim \frac{1}{(1-|w|^2)^{sq}}=h_s(w)^q.
\end{split}
\end{equation}
Combining \eqref{equa3.3} and \eqref{equa3.6}, by the generalized Schur's test in Lemma \ref{lemma2.3}, the operator $\mathcal{B}_{\mu,\beta,\delta}^{\xi}: L_{\alpha}^p(\mathbb{B}_n)\to L_{\alpha}^q(\mathbb{B}_n)$ is bounded. Then so is $\mathcal{C}_{\mu,\beta}^{\xi}: A_{\alpha}^p(\mathbb{B}_n)\to A_{\alpha}^q(\mathbb{B}_n)$.

$\bullet${\it Case 2. $b'>n+1+\delta'$}. In this case, we use item (4) in Lemma \ref{lemma2.2} to obtain
\begin{equation}\label{equa3.7}
\begin{split}
&\quad \int_{\mathbb{B}_n}\frac{(1-|z|^2)^{\alpha-sq}}{\left|1-t\langle z,\xi\rangle\right|^{\beta\gamma_2 q}\left|1-t\langle z,w\rangle\right|^{(n+1+\alpha)\gamma_2q}}{\rm d}v(z)\\
&\simeq \frac{1}{(1-t^2)^{b'-(n+1+\delta')}\left|1-t^2\langle w,\xi\rangle\right|^{c'}}+\frac{1}{(1-t^2|w|^2)^{c'-(n+1+\delta')}\left|1-t^2\langle w,\xi\rangle\right|^{b'}}\\
&\lesssim \frac{1}{(1-t)^{\beta\gamma_2 q+sq-(n+1+\alpha)}}\cdot\frac{1}{(1-t|w|)^{(n+1+\alpha)\gamma_2q}}.
\end{split}
\end{equation}
Since $\lambda+s/\gamma_2=\beta+s/\gamma_2-(n+1+\alpha)/(\gamma_2 q)+(n+1+\alpha)$, and $\mu$ is a $\lambda$-Carleson measure on $[0,1)$, by \eqref{equa3.4}, \eqref{equa3.7} and Lemma \ref{lemma2.6}, we obtain
\begin{equation}\label{equa3.8}
\begin{split}
&\quad \int_{\mathbb{B}_n}h_s(z)^q\left[\widetilde{K}_{\mu,\beta,\delta}^{\xi}(z,w)\right]^{\gamma_2q}{\rm d}v_{\alpha}(z)\\
&\lesssim \left(\int_{[0,1)}\frac{{\rm d}\mu(t)}{(1-t)^{\beta+s/\gamma_2-(n+1+\alpha)/(\gamma_2q)}(1-t|w|)^{n+1+\alpha}}\right)^{\gamma_2q}\lesssim \left(\frac{1}{(1-|w|^2)^{s/\gamma_2}}\right)^{\gamma_2 q}=h_s(w)^q.
\end{split}
\end{equation}
Combining \eqref{equa3.3} and \eqref{equa3.8}, by the generalized Schur's test again, we get $\mathcal{B}_{\mu,\beta,\delta}^{\xi}: L_{\alpha}^p(\mathbb{B}_n)\to L_{\alpha}^q(\mathbb{B}_n)$ is bounded. Hence  $\mathcal{C}_{\mu,\beta}^{\xi}: A_{\alpha}^p(\mathbb{B}_n)\to A_{\alpha}^q(\mathbb{B}_n)$ is also bounded.

$\bullet${\it Case 3. $b'=n+1+\delta'$}. In this case, we can choose $\varepsilon>0$ small enough such that $\varepsilon<\alpha+1-sq$. Then 
\begin{align*}
&\quad \int_{\mathbb{B}_n}\frac{(1-|z|^2)^{\alpha-sq}}{\left|1-t\langle z,\xi\rangle\right|^{\beta\gamma_2 q}\left|1-t\langle z,w\rangle\right|^{(n+1+\alpha)\gamma_2 q}}{\rm d}v(z)\\
&\lesssim \int_{\mathbb{B}_n}\frac{(1-|z|^2)^{\alpha-sq-\varepsilon}}{\left|1-t\langle z,\xi\rangle\right|^{\beta\gamma_2 q}\left|1-t\langle z,w\rangle\right|^{(n+1+\alpha)\gamma_2 q-\varepsilon}}{\rm d}v(z).
\end{align*}
Set $\delta''=\alpha-sq-\varepsilon=\delta'-\varepsilon$, $b''=\beta\gamma_2 q=b'$ and $c''=(n+1+\alpha)\gamma_2 q-\varepsilon=c'-\varepsilon$. From the choice of $\varepsilon$, we know that $\delta''>-1$ and 
\[
    b''-(n+1+\delta'')=b'-(n+1+\delta')+\varepsilon>0,
\]
\[
    c''-(n+1+\delta'')=c'-(n+1+\delta')>0.
\]
Thus, by an argument similar to that used in Case 2, we also get
\[
    \int_{\mathbb{B}_n}h_s(z)^q\left[\widetilde{K}_{\mu,\beta,\delta}^{\xi}(z,w)\right]^{\gamma_2 q}{\rm d}v_{\alpha}(z)\lesssim h_s(w)^q.
\]
By the generalized Schur's test again, $\mathcal{B}_{\mu,\beta,\delta}^{\xi}: L_{\alpha}^p(\mathbb{B}_n)\to L_{\alpha}^q(\mathbb{B}_n)$ is bounded, and so is $\mathcal{C}_{\mu,\beta}^{\xi}: A_{\alpha}^p(\mathbb{B}_n)\to A_{\alpha}^q(\mathbb{B}_n)$. The proof is now complete.
\end{proof}

\begin{proposition}\label{proposition3.2}
Let $\alpha>-1$, $\beta>0$ and $1\leq q<\infty$. Set 
\[
    \lambda=\beta+(n+1+\alpha)\Big(1-\frac{1}{q}\Big).
\]
If $\mu$ is a $\lambda$-Carleson measure on $[0,1)$, then $\mathcal{B}_{\mu,\beta,\delta}^{\xi}: L_{\alpha}^1(\mathbb{B}_n)\to L_{\alpha}^q(\mathbb{B}_n)$ is bounded for every $\delta>0$. Consequently, $\mathcal{C}_{\mu,\beta}^{\xi}: A_{\alpha}^1(\mathbb{B}_n)\to A_{\alpha}^q(\mathbb{B}_n)$ is bounded. 
\end{proposition}

\begin{proof}
Fix $\delta>0$. For any $f\in L_{\alpha}^1(\mathbb{B}_n)$, by \eqref{equa3.01}, 
\begin{align*}
\mathcal{B}_{\mu,\beta,\delta}^{\xi}f(z)=\frac{c_{\alpha+\delta}}{c_{\alpha}}\int_{\mathbb{B}_n}\left(\int_{[0,1)}\frac{f(w)(1-|w|^2)^{\delta}}{|1-t\langle z,\xi\rangle|^{\beta}\cdot|1-t\langle z,w\rangle|^{n+1+\alpha+\delta}}{\rm d}\mu(t)\right){\rm d}v_{\alpha}(w).
\end{align*}
Then we use Minkowski's inequality to obtain
\begin{equation}\label{equa3.10}
\begin{split}
&\quad\big\|\mathcal{B}_{\mu,\beta,\delta}^{\xi}f\big\|_{q,\alpha}\\
&\lesssim \left[\int_{\mathbb{B}_n}\left(\int_{\mathbb{B}_n}\int_{[0,1)}\frac{|f(w)|(1-|w|^2)^{\delta}{\rm d}\mu(t)}{|1-t\langle z,\xi\rangle|^{\beta}\cdot|1-t\langle z,w\rangle|^{n+1+\alpha+\delta}}{\rm d}v_{\alpha}(w)\right)^q{\rm d}v_{\alpha}(z)\right]^{1/q}\\
&\lesssim \int_{\mathbb{B}_n}|f(w)|(1-|w|^2)^{\delta}\int_{[0,1)}\left(\int_{\mathbb{B}_n}\frac{(1-|z|^2)^{\alpha}{\rm d}v(z)}{|1-t\langle z,\xi\rangle|^{\beta q}|1-t\langle z,w\rangle|^{(n+1+\alpha+\delta)q}}\right)^{1/q}{\rm d}\mu(t){\rm d}v_{\alpha}(w)
\end{split}
\end{equation}
Let $\delta'=\alpha$, $b'=\beta q$ and $c'=(n+1+\alpha+\delta)q$. Then $\delta'>-1$, $b'>0$ and 
\[
    c'-(n+1+\delta')=(n+1+\alpha)(q-1)+\delta q>0.
\]
Now we again discuss three cases: $b'<n+1+\delta'$, $b'>n+1+\delta'$ and $b'=n+1+\delta'$.

$\bullet${\it Case 1. $b'<n+1+\delta'$}. In this case, by item (3) of Lemma \ref{lemma2.2}, we get
\begin{align*}
&\quad \int_{\mathbb{B}_n}\frac{(1-|z|^2)^{\alpha}}{|1-t\langle z,\xi\rangle|^{\beta q}|1-t\langle z,w\rangle|^{(n+1+\alpha+\delta)q}}{\rm d}v(z)\\
&\simeq \frac{1}{\left(1-t^2|w|^2\right)^{c'-(n+1+\delta')}}\cdot\frac{1}{\left|1-t^2\langle w,\xi\rangle\right|^{b'}}\lesssim \frac{1}{(1-t|w|)^{(n+1+\alpha)(q-1)+\delta q+\beta q}}=\frac{1}{(1-t|w|)^{q\lambda+\delta q}}.
\end{align*}
Then by \eqref{equa3.10}, we get
\begin{align*}
\big\|\mathcal{B}_{\mu,\beta,\delta}^{\xi}f\big\|_{q,\alpha}\lesssim \int_{\mathbb{B}_n}|f(w)|(1-|w|^2)^{\delta}\left(\int_{[0,1)}\frac{1}{(1-t|w|)^{\lambda+\delta}}{\rm d}\mu(t)\right){\rm d}v_{\alpha}(w).
\end{align*}
Since $\mu$ is a $\lambda$-Carleson measure on $[0,1)$, Lemma \ref{lemma2.6} gives that
\[
    \int_{[0,1)}\frac{1}{(1-t|w|)^{\lambda+\delta}}{\rm d}\mu(t)\lesssim \frac{1}{(1-|w|^2)^{\delta}}.
\]
It follows that 
\[
    \big\|\mathcal{B}_{\mu,\beta,\delta}^{\xi}f\big\|_{q,\alpha}\lesssim \int_{\mathbb{B}_n}|f(w)|(1-|w|^2)^{\delta}\cdot \frac{1}{(1-|w|^2)^{\delta}}{\rm d}v_{\alpha}(w)=\|f\|_{1,\alpha}.
\]
This shows the boundedness of $\mathcal{B}_{\mu,\beta,\delta}^{\xi}: L_{\alpha}^1(\mathbb{B}_n)\to L_{\alpha}^q(\mathbb{B}_n)$, and hence $\mathcal{C}_{\mu,\beta}^{\xi}$ is bounded from $A_{\alpha}^1(\mathbb{B}_n)$ to $A_{\alpha}^q(\mathbb{B}_n)$.

$\bullet${\it Case 2. $b'>n+1+\delta'$}. In this case, we apply item (4) of Lemma \ref{lemma2.2} and modify the proof of Proposition \ref{proposition3.1} to obtain
\begin{equation}\label{equa3.11}
\int_{\mathbb{B}_n}\frac{(1-|z|^2)^{\alpha}{\rm d}v(z)}{|1-t\langle z,\xi\rangle|^{\beta q}|1-t\langle z,w\rangle|^{(n+1+\alpha+\delta)q}}\lesssim \frac{1}{(1-t)^{\beta q-(n+1+\alpha)}(1-t|w|)^{(n+1+\alpha+\delta)q}}.
\end{equation}
Since $n+1+\alpha+\delta=\lambda-(\beta-\frac{n+1+\alpha}{q})+\delta$, $\delta>0$, and $\mu$ is a $\lambda$-Carleson measure on $[0,1)$, by \eqref{equa3.10}, \eqref{equa3.11} and Lemma \ref{lemma2.6}, we obtain
\begin{align*}
\big\|\mathcal{B}_{\mu,\beta,\delta}^{\xi}f\big\|_{q,\alpha}&\lesssim \int_{\mathbb{B}_n}|f(w)|(1-|w|^2)^{\delta}\left(\int_{[0,1)}\frac{{\rm d}\mu(t)}{(1-t)^{\beta-(n+1+\alpha)/q}(1-t|w|)^{n+1+\alpha+\delta}}\right){\rm d}v_{\alpha}(w)\\
&\lesssim \int_{\mathbb{B}_n}|f(w)|(1-|w|^2)^{\delta}\cdot \frac{1}{(1-|w|^2)^{\delta}}{\rm d}v_{\alpha}(w)=\|f\|_{1,\alpha}.
\end{align*}
This shows that $\mathcal{B}_{\mu,\beta,\delta}^{\xi}: L_{\alpha}^1(\mathbb{B}_n)\to L_{\alpha}^q(\mathbb{B}_n)$ is bounded, and so $\mathcal{C}_{\mu,\beta}^{\xi}: A_{\alpha}^1(\mathbb{B}_n)\to A_{\alpha}^q(\mathbb{B}_n)$ is bounded.

$\bullet${\it Case 3. $b'=n+1+\delta'$}. In this case, similarly to the proof of Proposition \ref{proposition3.1}, we choose $\varepsilon>0$ small enough such that $\varepsilon<1+\alpha$. Then
\begin{align*}
&\quad \int_{\mathbb{B}_n}\frac{(1-|z|^2)^{\alpha}}{|1-t\langle z,\xi\rangle|^{\beta q}|1-t\langle z,w\rangle|^{(n+1+\alpha+\delta)q}}{\rm d}v(z)\\
&\lesssim \int_{\mathbb{B}_n}\frac{(1-|z|^2)^{\alpha-\varepsilon}}{|1-t\langle z,\xi\rangle|^{\beta q}|1-t\langle z,w\rangle|^{(n+1+\alpha+\delta)q-\varepsilon}}{\rm d}v(z).
\end{align*}
Let $\delta''=\alpha-\varepsilon=\delta'-\varepsilon$, $b''=\beta q=b'$ and $c''=(n+1+\alpha+\delta)q-\varepsilon=c'-\varepsilon$. Clearly, $\delta''>-1$ and
\[
    b''-(n+1+\delta'')=b'-(n+1+\delta')+\varepsilon>0,
\]
\[
    c''-(n+1+\delta'')=c'-(n+1+\delta')>0.
\]
Hence, by a similar argument as Case 2, we can also get
\[
    \big\|\mathcal{B}_{\mu,\beta,\delta}^{\xi}f\big\|_{q,\alpha}\lesssim \|f\|_{1,\alpha},\quad \forall f\in L_{\alpha}^1(\mathbb{B}_n).
\]
That is, $\mathcal{B}_{\mu,\beta,\delta}^{\xi}: L_{\alpha}^1(\mathbb{B}_n)\to L_{\alpha}^q(\mathbb{B}_n)$ is bounded, and $\mathcal{C}_{\mu,\beta}^{\xi}: A_{\alpha}^1(\mathbb{B}_n)\to A_{\alpha}^q(\mathbb{B}_n)$ is also bounded. The proof is now complete.
\end{proof}

 We next treat the case $0<p<1$ and $p\leq q<\infty$. We need to 
estimate the operator $\mathcal{C}_{\mu,\beta}^{\xi}$ on the atoms $a_{w,M}^{(p)}$ 
defined in \eqref{atomic-decomposition}. 
The following two lemmas are key steps in proving Theorem \ref{theorem1.1} 
for the case $0<p<1$ and $p\leq q<\infty$.

\begin{lemma}\label{lemma3.3}
Let \(\alpha>-1\), \(\beta>0\), \(0<p<1\), \(p<q\le\infty\), and
\(M>(n+1+\alpha)/p\). Suppose that \(\mu\) is
\(\lambda\)-Carleson, where
\[
    \lambda=\beta+(n+1+\alpha)\left(\frac{1}p-\frac{1}q\right),
    \qquad \frac{1}{\infty}=0.
\]
Then
\[
    \sup_{w\in\mathbb B_n}
    \|\mathcal C_{\mu,\beta}^{\xi}a_{w,M}^{(p)}\|_{q,\alpha}
    \lesssim
    \sup_{0\leq r<1}\frac{\widehat\mu(r)}{(1-r)^\lambda}.
\]
\end{lemma}

\begin{proof}
Fix \(z,w\in\mathbb B_n\). For \(k\ge0\), recall that
\(\delta_k=2^{-k}\), \(r_k=1-\delta_k\), and
\(I_k=[r_k,r_{k+1})\). If \(t\in I_k\), then
\(\delta_{k+1}<1-t\le\delta_k\), and hence
\(1-t\simeq\delta_k\). Applying \(|1-ta|\simeq(1-t)+|1-a|\) first with
\(a=\langle z,w\rangle\) and then with \(a=\langle z,\xi\rangle\), we obtain
\(|1-t\langle z,w\rangle|\simeq \delta_k+|1-\langle z,w\rangle|\) and
\(|1-t\langle z,\xi\rangle|\simeq \delta_k+|1-\langle z,\xi\rangle|\).
Using the definition of the atom and decomposing the measure integral
with respect to the intervals \(I_k\), we have
\begin{align*}
  |\mathcal C_{\mu,\beta}^{\xi}a_{w,M}^{(p)}(z)|
  &\le
  (1-|w|^2)^{M-(n+1+\alpha)/p}
  \sum_{k\ge0}
  \int_{I_k}
  \frac{{\rm d}\mu(t)}{|1-t\langle z,w\rangle|^M
   |1-t\langle z,\xi\rangle|^\beta}\\
  &\lesssim
  (1-|w|^2)^{M-(n+1+\alpha)/p}
  \sum_{k\ge0}
  \frac{\mu(I_k)}{(\delta_k+|1-\langle z,w\rangle|)^M
   (\delta_k+|1-\langle z,\xi\rangle|)^\beta}.
\end{align*}
Since \(I_k\subset[r_k,1)\), the \(\lambda\)-Carleson condition gives
\[
  \mu(I_k)
  \le \widehat\mu(r_k)
  \le
  \left(
    \sup_{0\le r<1}
    \frac{\widehat\mu(r)}{(1-r)^\lambda}
  \right)(1-r_k)^\lambda
  =
  \left(
    \sup_{0\le r<1}
    \frac{\widehat\mu(r)}{(1-r)^\lambda}
  \right)\delta_k^\lambda.
\]
Consequently,
\[
\begin{aligned}
  |\mathcal C_{\mu,\beta}^{\xi}a_{w,M}^{(p)}(z)|
  &\lesssim
  \left(
    \sup_{0\le r<1}
    \frac{\widehat\mu(r)}{(1-r)^\lambda}
  \right)
  (1-|w|^2)^{M-(n+1+\alpha)/p} 
  \sum_{k\ge0}
  \frac{\delta_k^\lambda}{(\delta_k+|1-\langle z,w\rangle|)^M
   (\delta_k+|1-\langle z,\xi\rangle|)^\beta}.
\end{aligned}
\]
Because
\((\delta_k+|1-\langle z,\xi\rangle|)^\beta\ge\delta_k^\beta\),
it remains to estimate
\[
  \sum_{k\ge0}
  \frac{\delta_k^{\lambda-\beta}}{(\delta_k+|1-\langle z,w\rangle|)^M}.
\]
Here
\[
  \lambda-\beta
  =(n+1+\alpha)\left(\frac{1}p-\frac{1}q\right)>0,
\]
and
\[
  M>\frac{n+1+\alpha}{p}\ge\lambda-\beta.
\]
Assume first that \(0<|1-\langle z,w\rangle|\le1\). Choose the unique
integer \(k_0\ge0\) such that
\[
  \delta_{k_0+1}
  <|1-\langle z,w\rangle|
  \le\delta_{k_0}.
\]
Since \(\delta_{k_0+1}=\delta_{k_0}/2\), this choice implies
\(
 \delta_{k_0}\simeq|1-\langle z,w\rangle|.
\)
For \(0\le k\le k_0\), we have
\(\delta_k\ge\delta_{k_0}\ge|1-\langle z,w\rangle|\), and therefore
\(
 \delta_k+|1-\langle z,w\rangle|\simeq\delta_k.
\)
It follows that
\[
  \sum_{k=0}^{k_0}
  \frac{\delta_k^{\lambda-\beta}}{(\delta_k+|1-\langle z,w\rangle|)^M}
  \lesssim \sum_{k=0}^{k_0}\delta_k^{\lambda-\beta-M}
  \lesssim \delta_{k_0}^{\lambda-\beta-M}
  \lesssim |1-\langle z,w\rangle|^{\lambda-\beta-M}.
\]
For \(k>k_0\), we have
\(\delta_k<|1-\langle z,w\rangle|\). Hence
\[
  \sum_{k>k_0}
  \frac{\delta_k^{\lambda-\beta}}{(\delta_k+|1-\langle z,w\rangle|)^M}
  \le \frac{1}{|1-\langle z,w\rangle|^M}\sum_{k>k_0}\delta_k^{\lambda-\beta}
  \lesssim \frac{\delta_{k_0+1}^{\lambda-\beta}}{|1-\langle z,w\rangle|^M}
  \lesssim |1-\langle z,w\rangle|^{\lambda-\beta-M}.
\]
Combining the two ranges gives
\[
  \sum_{k\ge0}
  \frac{\delta_k^{\lambda-\beta}}{(\delta_k+|1-\langle z,w\rangle|)^M}
  \lesssim
  |1-\langle z,w\rangle|^{\lambda-\beta-M}.
\]

If \(|1-\langle z,w\rangle|>1\), then
\(\delta_k+|1-\langle z,w\rangle|
 \ge |1-\langle z,w\rangle|\) for every \(k\), and therefore
\[
  \sum_{k\ge0}
  \frac{\delta_k^{\lambda-\beta}}{(\delta_k+|1-\langle z,w\rangle|)^M}
  \le \frac{1}{|1-\langle z,w\rangle|^M}\sum_{k\ge0}\delta_k^{\lambda-\beta}
  \lesssim |1-\langle z,w\rangle|^{\lambda-\beta-M}.
\]
Thus, in all cases,
\[
  |\mathcal C_{\mu,\beta}^{\xi}a_{w,M}^{(p)}(z)|
  \lesssim
  \left(
    \sup_{0\le r<1}
    \frac{\widehat\mu(r)}{(1-r)^\lambda}
  \right)
  \frac{(1-|w|^2)^{M-(n+1+\alpha)/p}}{|1-\langle z,w\rangle|^{M-\lambda+\beta}}.
\]
If \(q<\infty\), Lemma~\ref{lemma2.1} gives
\[
  \int_{\mathbb B_n}
  \left(
    \frac{(1-|w|^2)^{M-(n+1+\alpha)/p}}{|1-\langle z,w\rangle|^{M-\lambda+\beta}}
  \right)^q
  {\rm d}v_\alpha(z)\lesssim1,
\]
and therefore
\[
  \|\mathcal C_{\mu,\beta}^{\xi}a_{w,M}^{(p)}\|_{q,\alpha}
  \lesssim
  \sup_{0\le r<1}\frac{\widehat\mu(r)}{(1-r)^\lambda}.
\]
If \(q=\infty\), then \(\lambda-\beta=(n+1+\alpha)/p\), and
\(1-|w|^2\le2|1-\langle z,w\rangle|\). Hence the preceding pointwise estimate gives
\[
  |\mathcal C_{\mu,\beta}^{\xi}a_{w,M}^{(p)}(z)|
  \lesssim
  \left(\sup_{0\le r<1}\frac{\widehat\mu(r)}{(1-r)^\lambda}\right)
  \frac{(1-|w|^2)^{M-(n+1+\alpha)/p}}{|1-\langle z,w\rangle|^{M-(n+1+\alpha)/p}}
  \lesssim
  \sup_{0\le r<1}\frac{\widehat\mu(r)}{(1-r)^\lambda}.
\]
\end{proof}

We next consider the diagonal case $0<p=q<1$.

\begin{lemma}\label{lemma3.4}
Let \(\alpha>-1\), \(\beta>0\), \(0<p=q<1\), choose
\(M>(n+1+\alpha)/p\), and suppose that \(\mu\)
is \(\beta\)-Carleson. Then
\[
    \sup_{w\in\mathbb B_n}
    \|\mathcal C_{\mu,\beta}^{\xi}a_{w,M}^{(p)}\|_{p,\alpha}
    \lesssim
    \sup_{0\leq r<1}\frac{\widehat\mu(r)}{(1-r)^\beta}.
\]
\end{lemma}

\begin{proof}
Fix \(z,w\in\mathbb B_n\). Decomposing the measure integral over the intervals
\(I_k\), using \(1-t\simeq\delta_k\) for \(t\in I_k\), and applying the
\(\beta\)-Carleson condition, we obtain
\[
\begin{aligned}
  |\mathcal C_{\mu,\beta}^{\xi}a_{w,M}^{(p)}(z)|
  &\lesssim
  \left(\sup_{0\le r<1}\frac{\widehat\mu(r)}{(1-r)^\beta}\right)
  (1-|w|^2)^{M-(n+1+\alpha)/p}
  \sum_{k\ge0}
  \frac{\delta_k^\beta}{(\delta_k+|1-\langle z,w\rangle|)^M
        (\delta_k+|1-\langle z,\xi\rangle|)^\beta}.
\end{aligned}
\]

\noindent If \(|1-\langle z,\xi\rangle|\ge |1-\langle z,w\rangle|\), we split the sum into
\(\delta_k\ge |1-\langle z,w\rangle|\) and \(\delta_k<|1-\langle z,w\rangle|\). Then
\[
\begin{aligned}
  &\sum_{k\ge0}
  \frac{\delta_k^\beta}{(\delta_k+|1-\langle z,w\rangle|)^M
        (\delta_k+|1-\langle z,\xi\rangle|)^\beta} \\
  &\quad\lesssim
  \sum_{\delta_k\ge |1-\langle z,w\rangle|}\delta_k^{-M}
  +\frac{1}{|1-\langle z,w\rangle|^M|1-\langle z,\xi\rangle|^\beta}
  \sum_{\delta_k<|1-\langle z,w\rangle|}\delta_k^\beta
  \lesssim |1-\langle z,w\rangle|^{-M}.
\end{aligned}
\]

\noindent If \(|1-\langle z,\xi\rangle|<|1-\langle z,w\rangle|\), we split the sum into the three ranges
\(\delta_k\ge |1-\langle z,w\rangle|\),
\(|1-\langle z,\xi\rangle|\le\delta_k<|1-\langle z,w\rangle|\), and
\(\delta_k<|1-\langle z,\xi\rangle|\). Thus,
\[
\begin{aligned}
  &\sum_{k\ge0}
  \frac{\delta_k^\beta}{(\delta_k+|1-\langle z,w\rangle|)^M
        (\delta_k+|1-\langle z,\xi\rangle|)^\beta} \\
  &\quad\lesssim
  \sum_{\delta_k\ge |1-\langle z,w\rangle|}\delta_k^{-M}
  +\frac{1}{|1-\langle z,w\rangle|^M}
   \left(1+\log\frac{|1-\langle z,w\rangle|}{|1-\langle z,\xi\rangle|}\right)
  +\frac{1}{|1-\langle z,w\rangle|^M|1-\langle z,\xi\rangle|^\beta}
   \sum_{\delta_k<|1-\langle z,\xi\rangle|}\delta_k^\beta \\
  &\quad\lesssim
  |1-\langle z,w\rangle|^{-M}
  \left(1+\log\frac{|1-\langle z,w\rangle|}{|1-\langle z,\xi\rangle|}\right).
\end{aligned}
\]
Hence, in both cases,
\[
  \sum_{k\ge0}
  \frac{\delta_k^\beta}{(\delta_k+|1-\langle z,w\rangle|)^M
        (\delta_k+|1-\langle z,\xi\rangle|)^\beta}
  \lesssim
  \frac{1}{|1-\langle z,w\rangle|^M}
  \left(1+\log^+
  \frac{|1-\langle z,w\rangle|}{|1-\langle z,\xi\rangle|}\right).
\]

Choose \(0<\varepsilon<(\alpha+1)/p\). Since
\(1+\log^+x\lesssim_\varepsilon1+x^\varepsilon\) for \(x>0\),
\[
\begin{aligned}
  |\mathcal C_{\mu,\beta}^{\xi}a_{w,M}^{(p)}(z)|
  &\lesssim
  \left(\sup_{0\le r<1}\frac{\widehat\mu(r)}{(1-r)^\beta}\right)
  (1-|w|^2)^{M-(n+1+\alpha)/p}\\
  &\qquad\times
  \left[
    \frac{1}{|1-\langle z,w\rangle|^M}
    +\frac{1}{|1-\langle z,w\rangle|^{M-\varepsilon}
             |1-\langle z,\xi\rangle|^\varepsilon}
  \right].
\end{aligned}
\]
Because \(0<p<1\), \((x+y)^p\le x^p+y^p\). 
Applying Lemma~2.1 with exponent $pM$, we obtain
\[
  (1-|w|^2)^{pM-(n+1+\alpha)}
  \int_{\mathbb B_n}
  \frac{(1-|z|^2)^\alpha}{|1-\langle z,w\rangle|^{pM}}\,{\rm d}v(z)
  \lesssim1.
\]
Moreover, \(|1-\langle z,\xi\rangle|\ge 1-|\langle z,\xi\rangle|\ge 1-|z|=\frac{1-|z|^2}{1+|z|}\ge\frac{1-|z|^2}{2}\). Hence
\[
  (1-|w|^2)^{pM-(n+1+\alpha)}
  \int_{\mathbb B_n}
  \frac{(1-|z|^2)^{\alpha-p\varepsilon}}{|1-\langle z,w\rangle|^{p(M-\varepsilon)}}\,{\rm d}v(z)
  \lesssim1
\]
by Lemma~\ref{lemma2.1}, since
\(\alpha-p\varepsilon>-1\) and
\(p(M-\varepsilon)>n+1+\alpha-p\varepsilon\). Therefore
\[
  \|\mathcal C_{\mu,\beta}^{\xi}a_{w,M}^{(p)}\|_{p,\alpha}
  \lesssim
  \sup_{0\le r<1}\frac{\widehat\mu(r)}{(1-r)^\beta},
\]
uniformly in \(w\).
\end{proof}

Next we are ready to prove the
sufficient condition for the boundedness of $\mathcal C_{\mu,\beta}^{\xi}$ in the range $0<p<1$ and $p\le q\le\infty$.

\begin{proposition}\label{proposition3.5}
Let \(\alpha>-1\), \(\beta>0\), \(0<p<1\), and \(p\le q\le\infty\). Set
\[
    \lambda=\beta+(n+1+\alpha)\left(\frac{1}p-\frac{1}q\right),
    \qquad \frac{1}{\infty}=0.
\]
If \(\mu\) is a \(\lambda\)-Carleson measure, then
\[
  \|\mathcal C_{\mu,\beta}^{\xi}f\|_{q,\alpha}
  \lesssim
  \left(
    \sup_{0\le r<1}
    \frac{\widehat\mu(r)}{(1-r)^\lambda}
  \right)
  \|f\|_{p,\alpha},
  \qquad f\in A_\alpha^p.
\]
In particular,
\(\mathcal C_{\mu,\beta}^{\xi}:A_\alpha^p(\mathbb B_n)\to
A_\alpha^q(\mathbb B_n)\) is bounded.
\end{proposition}

\begin{proof}
Choose \(M>(n+1+\alpha)/p\). Lemma~\ref{lemma3.3}
for \(p<q\le\infty\) and Lemma~\ref{lemma3.4} with
\(\lambda=\beta\) for \(p=q<1\) give
\[
  \sup_{w\in\mathbb B_n}
  \|\mathcal C_{\mu,\beta}^{\xi}a_{w,M}^{(p)}\|_{q,\alpha}
  \lesssim
  \sup_{0\le r<1}
  \frac{\widehat\mu(r)}{(1-r)^\lambda}.
\]
Choose an atomic representation from
Proposition~\ref{proposition2.4} such that
\[
  f=\sum_{j=1}^\infty c_ja_{w_j,M}^{(p)},
  \qquad
  \sum_{j=1}^\infty|c_j|^p
  \lesssim\|f\|_{p,\alpha}^p.
\]
We first prove that
\[
  \sum_{j=1}^\infty c_j\mathcal C_{\mu,\beta}^{\xi}a_{w_j,M}^{(p)}
  \quad\text{converges in }A_\alpha^q.
\]

Suppose \(q<1\). Since \(p\le q\), the inclusion
\(\ell^p\subset\ell^q\) holds and
\(
 \sum_j|c_j|^q
 \le(\sum_j|c_j|^p)^{q/p}.
\)
For integers \(L>N\), the \(q\)-subadditivity of the integral
quasi-norm gives
\[
  \left\|\sum_{j=N+1}^Lc_j\mathcal C_{\mu,\beta}^{\xi}a_{w_j,M}^{(p)}\right\|_{q,\alpha}^q
  \le \sum_{j=N+1}^L|c_j|^q\|\mathcal C_{\mu,\beta}^{\xi}a_{w_j,M}^{(p)}\|_{q,\alpha}^q
  \lesssim \left(\sup_{0\le r<1}\frac{\widehat\mu(r)}{(1-r)^\lambda}\right)^q\sum_{j=N+1}^L|c_j|^q.
\]
Hence
\[
  \lim_{N\to\infty}\sup_{L>N}\left\|\sum_{j=N+1}^Lc_j\mathcal C_{\mu,\beta}^{\xi}a_{w_j,M}^{(p)}\right\|_{q,\alpha}=0.
\]

Suppose \(1\le q\le\infty\). Since \(0<p<1\),
\[
  \sum_{j=1}^\infty|c_j|
  \le
  \left(\sum_{j=1}^\infty|c_j|^p\right)^{1/p},
\]
which follows from
\((\sum_j|c_j|)^p\le\sum_j|c_j|^p\). By the triangle inequality in
\(A_\alpha^q\),
\[
  \left\|\sum_{j=N+1}^Lc_j\mathcal C_{\mu,\beta}^{\xi}a_{w_j,M}^{(p)}\right\|_{q,\alpha}
  \le \sum_{j=N+1}^L|c_j|\|\mathcal C_{\mu,\beta}^{\xi}a_{w_j,M}^{(p)}\|_{q,\alpha}
  \lesssim \left(\sup_{0\le r<1}\frac{\widehat\mu(r)}{(1-r)^\lambda}\right)\sum_{j=N+1}^L|c_j|.
\]
Hence
\[
  \lim_{N\to\infty}\sup_{L>N}\left\|\sum_{j=N+1}^Lc_j\mathcal C_{\mu,\beta}^{\xi}a_{w_j,M}^{(p)}\right\|_{q,\alpha}=0.
\]
Therefore, by completeness of \(A_\alpha^q\), there exists \(g\in A_\alpha^q\) such that
\[
  g=\sum_{j=1}^\infty c_j\mathcal C_{\mu,\beta}^{\xi}a_{w_j,M}^{(p)}
  \quad\text{in }A_\alpha^q.
\]

Setting \(N=0\) and passing to \(L\to\infty\) gives
\[
  \|g\|_{q,\alpha}
  \lesssim
  \left(
    \sup_{0\le r<1}
    \frac{\widehat\mu(r)}{(1-r)^\lambda}
  \right)
  \left(\sum_j|c_j|^p\right)^{1/p}
  \lesssim
  \left(
    \sup_{0\le r<1}
    \frac{\widehat\mu(r)}{(1-r)^\lambda}
  \right)
  \|f\|_{p,\alpha}.
\]
It remains to verify that \(g=\mathcal C_{\mu,\beta}^{\xi}f\). Since \(\mu\) is a
\(\lambda\)-Carleson measure, 
\[
  \mu([0,1))=\widehat\mu(0)\le \sup_{0\le r<1}\frac{\widehat\mu(r)}{(1-r)^\lambda}<\infty.
\]

Fix \(0<R<1\) and \(|z|\le R\). For every \(0\le t<1\),
\(|tz|\le R\), and hence
\[
  |1-\langle tz,w_j\rangle|
  \ge1-|tz||w_j|
  \ge1-R.
\]
Since \(M-(n+1+\alpha)/p>0\),
\((1-|w_j|^2)^{M-(n+1+\alpha)/p}\le1\), and therefore
\[
  |a_{w_j,M}^{(p)}(tz)|\le(1-R)^{-M}.
\]
Also,
\[
  |1-t\langle z,\xi\rangle|\ge1-t|z|\ge1-R.
\]
Since \(\ell^p\subset\ell^1\) and \(\mu\) is finite,
\[
  \sum_{j=1}^\infty\int_{[0,1)}\frac{|c_j|\,|a_{w_j,M}^{(p)}(tz)|}{|1-t\langle z,\xi\rangle|^\beta}\,{\rm d}\mu(t)
  \le \mu([0,1))(1-R)^{-M-\beta}\sum_{j=1}^\infty|c_j|<\infty.
\]
Moreover,
\[
  \sup_{|z|\le R}\sum_{j=1}^\infty\int_{[0,1)}\frac{|c_j|\,|a_{w_j,M}^{(p)}(tz)|}{|1-t\langle z,\xi\rangle|^\beta}\,{\rm d}\mu(t)<\infty.
\]
Thus Tonelli's theorem and uniform absolute convergence allow us to interchange
\(\sum_j\) and \(\int_{[0,1)}\):
\[
  \mathcal C_{\mu,\beta}^{\xi}f(z)=\int_{[0,1)}\frac{\sum_jc_ja_{w_j,M}^{(p)}(tz)}{(1-t\langle z,\xi\rangle)^\beta}\,{\rm d}\mu(t)
  =\sum_jc_j\mathcal C_{\mu,\beta}^{\xi}a_{w_j,M}^{(p)}(z).
\]
Since point evaluation at \(z\) is continuous on \(A_\alpha^q\),
\[
  \sum_{j=1}^\infty c_j\mathcal C_{\mu,\beta}^{\xi}a_{w_j,M}^{(p)}(z)=g(z).
\]
Hence, for every \(z\in\mathbb B_n\),
\[
  \mathcal C_{\mu,\beta}^{\xi}f(z)=\sum_{j=1}^\infty c_j\mathcal C_{\mu,\beta}^{\xi}a_{w_j,M}^{(p)}(z)=g(z).
\]
Consequently,
\[
  \|\mathcal C_{\mu,\beta}^{\xi}f\|_{q,\alpha}
  \lesssim
  \left(\sup_{0\le r<1}\frac{\widehat\mu(r)}{(1-r)^\lambda}\right)\|f\|_{p,\alpha}.
\]

\end{proof}

 Now we turn to the necessity. In fact, we do not require that $1\leq p\leq q$ for the necessity.
\begin{proposition}  \label{proposition3.6}
  Let $\alpha>-1$, $\beta>0$ and $0<p,q<\infty$. Suppose $\mu$ is a positive Borel measure on $[0,1)$. Set 
\[
    \lambda=\beta+(n+1+\alpha)\Big(\frac{1}{p}-\frac{1}{q}\Big)>0.
\]
If $\mathcal{C}_{\mu,\beta}^{\xi}: A_{\alpha}^p(\mathbb{B}_n)\to A_{\alpha}^q(\mathbb{B}_n)$ is bounded, then $\mu$ is a $\lambda$-Carleson measure on $[0,1)$.
\end{proposition}

\begin{proof}
For any $r\in [0,1)$, let 
\[
    f_{r,p}(z)=\frac{(1-r^2)^{(n+1+\alpha)/p}}{(1-r\langle z,\xi\rangle)^{2(n+1+\alpha)/p}}, \quad z\in\mathbb B_n.
\]
By Lemma \ref{lemma2.1}, $f_{r,p}\in A_{\alpha}^p(\mathbb{B}_n)$ and $\|f_{r,p}\|_{p,\alpha}\simeq 1$. If $\mathcal{C}_{\mu,\,\beta}^{\xi}$ is bounded from $A_{\alpha}^p(\mathbb{B}_n)$ to $A_{\alpha}^q(\mathbb{B}_n)$, then 
\[
    \sup_{r\in [0,1)}\big\|\mathcal{C}_{\mu,\,\beta}^{\xi}f_{r,p}\big\|_{q,\alpha}<\infty.
\]
On the other hand, by the pointwise estimate in Bergman spaces (see \cite[Theorem 2.1]{Zhu2005}), we get
\begin{align*}
\big\|\mathcal{C}_{\mu,\,\beta}^{\xi}f_{r,p}\big\|_{q,\alpha}&\gtrsim (1-r^2)^{(n+1+\alpha)/q}\big|(\mathcal{C}_{\mu,\,\beta}^{\xi}f_{r,p})(r\xi)\big|\\
&=(1-r^2)^{(n+1+\alpha)/q}\int_{[0,1)}\frac{(1-r^2)^{(n+1+\alpha)/p}}{(1-r^2t)^{2(n+1+\alpha)/p}(1-tr)^{\beta}}{\rm d}\mu(t)\\
&\geq (1-r^2)^{(n+1+\alpha)(1/p+1/q)}\int_{[r,1)}\frac{1}{(1-r^2t)^{2(n+1+\alpha)/p}(1-tr)^{\beta}}{\rm d}\mu(t).
\end{align*}
Since $1-r^2t\simeq 1-rt\simeq 1-r$ for $r<t<1$, we get
\begin{align*}
\big\|\mathcal{C}_{\mu,\beta}^{\xi}f_{r,p}\big\|_{q,\alpha}\gtrsim \frac{\hat{\mu}(r)}{(1-r)^{\beta+(n+1+\alpha)(1/p-1/q)}}=\frac{\hat{\mu}(r)}{(1-r)^{\lambda}}.
\end{align*}
Therefore,
\[
    \sup_{r\in [0,1)}\frac{\hat{\mu}(r)}{(1-r)^{\lambda}}<\infty.
\]
That is, $\mu$ is a $\lambda$-Carleson measure on $[0,1)$.
\end{proof}

\begin{proof}[{\bf Proof of Theorem 1.1}] 
The sufficiency follows from Proposition \ref{proposition3.1} when
$1<p\leq q<\infty$, from Proposition \ref{proposition3.2} when   $p=1$,
and from Proposition \ref{proposition3.5} when $0<p<1$. The necessity
follows from Proposition \ref{proposition3.6}. 
\end{proof}

\section{Proof of Theorem 1.2}

In this section, we present the proof of Theorem \ref{theorem1.2}. We first prove the sufficiency. To this end, some lemmas are provided.

Given a positive Borel measure $\mu$ on $[0,1)$, define 
\[
    V_{\mu}^{\xi}(z)=\int_{[0,1)}\frac{{\rm d}\mu(t)}{|1-t\langle z,\xi\rangle|^{\beta}},\quad z\in\mathbb{B}_n,
\]
and
\[
    V_{\mu}(u)=\int_{[0,1)}\frac{{\rm d}\mu(t)}{|1-tu|^{\beta}},\quad u\in\mathbb{D}.
\]

\begin{lemma}\label{lemma4.1}
Let $\alpha>-1$, $\theta>0$ and $\mu$ be a positive Borel measure on $[0,1)$. Then
\begin{equation}\label{equa4.1}
\int_{\mathbb{B}_n}\left|V_{\mu}^{\xi}(z)\right|^{\theta}{\rm d}v_{\alpha}(z)\simeq \int_{\mathbb{D}}\left|V_{\mu}(u)\right|^{\theta}(1-|u|^2)^{n+\alpha-1}{\rm d}A(u),
\end{equation}
where ${\rm d}A$ is the normalized area measure on $\mathbb{D}$.
\end{lemma}

\begin{proof}
If $n=1$, then \eqref{equa4.1} holds by the rotational invariance of $(1-|u|^2)^{\alpha}{\rm d}A(u)$. 

Now we assume $n\geq 2$. Choose an appropriate unitary transformation $U$ of $\mathbb{C}^n$ such that $U\xi=e_1$. By the unitary invariance of ${\rm d}v_{\alpha}$, we have
\begin{equation}\label{equa4.2}
\begin{split}
&\quad \int_{\mathbb{B}_n}\left|V_{\mu}^{\xi}(z)\right|^{\theta}{\rm d}v_{\alpha}(z)=\int_{\mathbb{B}_n}\left|V_{\mu}(z_1)\right|^{\theta}{\rm d}v_{\alpha}(z)\\
&\simeq \int_{0}^1r^{2n-1}(1-r^2)^{\alpha}{\rm d}r\int_{\partial\mathbb{B}_n}\left|V_{\mu}(r\eta_1)\right|^{\theta}{\rm d}\sigma(\eta).
\end{split}
\end{equation}
Here, ${\rm d}\sigma$ is the normalized surface measure on $\partial\mathbb{B}_n$. By \cite[Lemma 1.9]{Zhu2005}, we have
\begin{align*}
\int_{\partial\mathbb{B}_n}\left|V_{\mu}(r\eta_1)\right|^{\theta}{\rm d}\sigma(\eta)=(n-1)\int_{\mathbb{D}}(1-|u|^2)^{n-2}\left|V_{\mu}(ru)\right|^{\theta}{\rm d}A(u).
\end{align*}
Substituting this into \eqref{equa4.2} yields
\begin{align}\label{equa4.003}
\int_{\mathbb{B}_n}\left|V_{\mu}^{\xi}(z)\right|^{\theta}{\rm d}v_{\alpha}(z)\simeq \int_{0}^1 r^{2n-1}(1-r^2)^{\alpha}{\rm d}r\int_{\mathbb{D}}(1-|u|^2)^{n-2}\left|V_{\mu}(ru)\right|^{\theta}{\rm d}A(u).
\end{align}
Making the change of variables $w=ru$, we have $|w|<r$ and ${\rm d}A(u)=r^{-2}{\rm d}A(w)$. Using Fubini's theorem, we obtain
\begin{equation}\label{equa4.004}
\begin{split}
&\quad \int_{0}^1 r^{2n-1}(1-r^2)^{\alpha}{\rm d}r\int_{\mathbb{D}}(1-|u|^2)^{n-2}\left|V_{\mu}(ru)\right|^{\theta}{\rm d}A(u)\\
&= \int_{0}^1 r(1-r^2)^{\alpha}{\rm d}r\int_{|w|<r}\big(r^2-|w|^2\big)^{n-2}\left|V_{\mu}(w)\right|^{\theta}{\rm d}A(w)\\
&=\int_{\mathbb{D}}\left|V_{\mu}(w)\right|^{\theta}\left(\int_{|w|}^{1}r(1-r^2)^{\alpha}(r^2-|w|^2)^{n-2}{\rm d}r\right){\rm d}A(w).
\end{split}
\end{equation}
It remains to estimate the inner integral. Taking $r^2=|w|^2+(1-|w|^2)x$, we obtain
\begin{equation}\label{equa4.005}
\begin{split}
&\quad\int_{|w|}^1 r(1-r^2)^{\alpha}\left(r^2-|w|^2\right)^{n-2}{\rm d}r\\
&\simeq (1-|w|^2)^{n+\alpha-1}\int_{0}^1 (1-x)^{\alpha}x^{n-2}{\rm d}x
\simeq (1-|w|^2)^{n+\alpha-1},
\end{split}
\end{equation}
since $\alpha>-1$ and $n-2>-1$. Therefore, combining \eqref{equa4.003}, \eqref{equa4.004} and \eqref{equa4.005},
\[
    \int_{\mathbb{B}_n}\left|V_{\mu}^{\xi}(z)\right|^{\theta}{\rm d}v_{\alpha}(z)\simeq \int_{\mathbb{D}}\left|V_{\mu}(w)\right|^{\theta}(1-|w|^2)^{n+\alpha-1}{\rm d}A(w).
\]
The proof is complete.
\end{proof}

\begin{lemma}\label{lemma4.2}
For $k\geq 1$, let $E_k=\{u\in\mathbb{D}: \delta_{k+1}<|1-u|\leq \delta_k\}$. Then 
\begin{equation*}
\int_{E_k}(1-|u|^2)^{n+\alpha-1}{\rm d}A(u)\simeq \delta_{k}^{n+1+\alpha}.
\end{equation*}
\end{lemma}

\begin{proof}
Put $w=1-u$, then $u\in E_k$ means that
\[
    \frac{1}{2}\delta_k\leq |w|\leq \delta_k \quad \text{and}\quad |1-w|<1.
\]
It follows that $|w|^2<2{\rm Re}w$. Writing $w=r{\rm e}^{i\theta}$, then $\cos\theta>r/2$ and
\[
    1-|u|^2=1-|1-w|^2=2{\rm Re}w-|w|^2=r(2\cos\theta-r).
\]
Therefore, 
\begin{align*}
\int_{E_k}(1-|u|^2)^{n+\alpha-1}{\rm d}A(u)&\simeq \int_{\delta_k/2}^{\delta_k}\int_{\cos\theta>r/2}\left[r(2\cos\theta-r)\right]^{n+\alpha-1}r{\rm d}\theta\,{\rm d}r\\
&=\int_{\delta_k/2}^{\delta_k}r^{n+\alpha}{\rm d}r\int_{\cos\theta>r/2}(2\cos\theta-r)^{n+\alpha-1}{\rm d}\theta.
\end{align*}
Since $n+\alpha-1>-1$, the integral
\[
    \int_{\cos\theta>r/2}(2\cos\theta-r)^{n+\alpha-1}{\rm d}\theta
\]
is bounded above and below for $0<r\leq \frac{1}{2}$. Thus,
\begin{equation*}
\int_{E_k}(1-|u|^2)^{n+\alpha-1}{\rm d}A(u)\simeq \int_{\delta_k/2}^{\delta_k}r^{n+\alpha}{\rm d}r\simeq \delta_k^{n+1+\alpha}.\qedhere
\end{equation*}
\end{proof}

\begin{lemma}\label{lemma4.3}
Let $\alpha>-1$, $\beta>0$ and   $\theta>0$. Set
\[
    \lambda=\beta-\frac{n+1+\alpha}{\theta}.
\]
Suppose that $\mu$ is a finite positive Borel measure on $[0,1)$. Then $V_{\mu}^{\xi}\in L_{\alpha}^{\theta}(\mathbb{B}_n)$ if and only if
\begin{equation}\label{equa4.3}
\int_{0}^1\Big(\frac{\hat{\mu}(r)}{(1-r)^{\lambda}}\Big)^{\theta}\frac{{\rm d}r}{1-r}<\infty.
\end{equation}
\end{lemma}

\begin{proof}
By Lemma \ref{lemma2.7}, the integral condition in \eqref{equa4.3} is equivalent to
\[
    \sum_{k=1}^{\infty}\left(\frac{\hat{\mu}(r_k)}{\delta_k^{\lambda}}\right)^{\theta}<\infty,
\]
where $r_k=1-2^{-k}$, $\delta_k=2^{-k}$. Denote by $h_k=\hat{\mu}(r_k)/\delta_k^{\lambda}$. We only need to establish the equivalence between the condition   $\{h_k\}\in \ell^{\theta}$  and the condition $V_{\mu}^{\xi}\in L_{\alpha}^{\theta}(\mathbb{B}_n)$.

$\bullet$(i) Assume first that   $\{h_k\}\in \ell^{\theta}$. Let
\[
    V_{\mu,0}(u)=\int_{[0,\frac{1}{2})}\frac{{\rm d}\mu(t)}{|1-tu|^{\beta}},\quad V_{\mu,1}(u)=\int_{[\frac{1}{2}, 1)}\frac{{\rm d}\mu(t)}{|1-tu|^{\beta}},\quad u\in\mathbb{D}.
\]
By Lemma \ref{lemma4.1}, 
\begin{equation*}
\int_{\mathbb{B}_n}\left|V_{\mu}^{\xi}(z)\right|^{\theta}{\rm d}v_{\alpha}(z)\simeq \int_{\mathbb{D}}\left(|V_{\mu,0}(u)|^{\theta}+|V_{\mu,1}(u)|^{\theta}\right)(1-|u|^2)^{n+\alpha-1}{\rm d}A(u).
\end{equation*}
Clearly, 
\[
    \int_{\mathbb{D}}\left|V_{\mu,0}(u)\right|^{\theta}(1-|u|^2)^{n+\alpha-1}{\rm d}A(u)<\infty,
\]
since $\mu$ is finite and $n+\alpha-1>-1$. Hence we need to prove the finiteness of the integral $\int_{\mathbb{D}}|V_{\mu,1}(u)|^{\theta}(1-|u|^2)^{n+\alpha-1}{\rm d}A(u)$.

For $j\geq 1$, write $E_j=\{u\in\mathbb{D}: \delta_{j+1}<|1-u|\leq \delta_j\}$, $I_j=[r_j,r_{j+1})$ and $\mu_{j}=\mu(I_j)$. If $u\in E_k$ and $t\in I_j\, (k,j\geq 1)$, then $1-t\simeq \delta_j$, $|1-u|\simeq \delta_k$ and $t\geq 1/2$. Since ${\rm Re}(1-u)>0$, we get
\begin{align}\label{equa4.4}
|1-tu|^2=\left|1-t+t(1-u)\right|^2\geq (1-t)^2+t^2|1-u|^2\simeq \delta_j^2+\delta_k^2.
\end{align}
Hence, $|1-tu|\gtrsim \max\{\delta_j,\delta_k\}$. Therefore, when $u\in E_k$, we obtain
\begin{align*}
V_{\mu,1}(u)=\sum_{j=1}^{\infty}\int_{I_j}\frac{{\rm d}\mu(t)}{|1-tu|^{\beta}}\lesssim \sum_{j=1}^{\infty}\frac{\mu_j}{\max\{\delta_j,\delta_k\}^{\beta}}=\sum_{j<k}\frac{\mu_j}{\delta_j^{\beta}}+\sum_{j\geq k}\frac{\mu_j}{\delta_k^{\beta}}.
\end{align*}
Note that $\mu_j\leq \hat{\mu}(r_j)=h_j\cdot \delta_j^{\lambda}$, so we have
\begin{equation*}
\begin{split}
\sum_{j<k}\frac{\mu_j}{\delta_j^{\,\beta}}\leq \sum_{j<k}h_j\cdot \delta_j^{\,\lambda-\beta}=\sum_{j<k}h_j\cdot \delta_j^{\,-(n+1+\alpha)/\theta}=2^{k(n+1+\alpha)/\theta} \sum_{j<k}2^{(j-k)(n+1+\alpha)/\theta}\cdot h_j.
\end{split}
\end{equation*}
Besides,
\begin{equation*}
\sum_{j\geq k}\frac{\mu_j}{\delta_k^{\,\beta}}=\frac{\hat{\mu}(r_k)}{\delta_k^{\,\beta}}=h_k\cdot\delta_k^{\,\lambda-\beta}=2^{k(n+1+\alpha)/\theta}\cdot h_k.
\end{equation*}
Thus, when $u\in E_k$, we get
\begin{equation*}
V_{\mu,1}(u)\lesssim 2^{k(n+1+\alpha)/\theta}\cdot \widetilde{h}_k,
\end{equation*}
where $\widetilde{h}_k=\sum_{j\leq k}2^{(j-k)(n+1+\alpha)/\theta}\cdot h_j$. Then using Lemma \ref{lemma4.2}, we obtain
\begin{align*}
\int_{E_k}\left|V_{\mu,1}(u)\right|^{\theta}(1-|u|^2)^{n+\alpha-1}{\rm d}A(u)\lesssim \widetilde{h}_k^{\,\theta}.
\end{align*}

Let $h_0=0$ and $\mathbf{h}=\{h_k\}_{k\geq 0}$. Choose a sequence $\mathbf{a}=\{a_k\}$ such that
\[
     a_k=2^{-k(n+1+\alpha)/\theta},\quad k\geq 0.
\]
Then 
\[
    \widetilde{h}_k=\sum_{j=0}^{k}a_j\cdot h_{k-j}.
\]
  If  $\theta\geq 1$, the discrete Young inequality on $\ell^{\theta}$ yields
\[
    \left\|\{\widetilde{h}_k\}\right\|_{\ell^{\theta}}=\big\|\mathbf{a}*\mathbf{h}\big\|_{\ell^{\theta}}\leq \big\|\mathbf{a}\big\|_{\ell^1}\cdot \big\|\mathbf{h}\big\|_{\ell^{\theta}}\lesssim \big\|\mathbf{h}\big\|_{\ell^{\theta}}.
\]
 If $0<\theta<1$, subadditivity gives, with $A=n+1+\alpha>0$,
\[
    \widetilde h_k^{\,\theta}
    \leq \sum_{j=1}^{k}2^{-(k-j)A}h_j^{\,\theta}.
\]
Consequently,
\[
    \sum_{k=1}^{\infty}\widetilde h_k^{\,\theta}
    \leq \sum_{j=1}^{\infty}h_j^{\,\theta}
          \sum_{m=0}^{\infty}2^{-mA}
    \lesssim \sum_{j=1}^{\infty}h_j^{\,\theta}.
\]
 Consequently, 
\begin{equation}\label{equa4.5}
\begin{split}
&\quad \int_{\mathbb{D}\cap \big\{|1-u|<1/2\big\}}\left|V_{\mu,1}(u)\right|^{\theta}(1-|u|^2)^{n+\alpha-1}{\rm d}A(u)\\
&=\sum_{k=1}^{\infty}\int_{E_k}\left|V_{\mu,1}(u)\right|^{\theta}(1-|u|^2)^{n+\alpha-1}{\rm d}A(u)\lesssim \sum_{k=1}^{\infty}\widetilde{h}_k^{\,\theta}\lesssim \sum_{k=1}^{\infty}h_k^{\,\theta}<\infty.
\end{split}
\end{equation}
On the other hand, if $|1-u|\geq 1/2$ and $t\geq 1/2$, then
\[
    |1-tu|\gtrsim (1-t)+t|1-u|\geq \frac{1}{4}.
\]
It follows that 
\begin{align*}
&\quad\int_{\mathbb{D}\cap\big\{|1-u|\geq 1/2\big\}}\left|V_{\mu,1}(u)\right|^{\theta}(1-|u|^2)^{n+\alpha-1}{\rm d}A(u)\\
&\lesssim \left[4^{\beta}\mu([1/2,1))\right]^{\theta}\cdot\int_{\mathbb{D}}(1-|u|^2)^{n+\alpha-1}{\rm d}A(u)<\infty.
\end{align*}
This, together with \eqref{equa4.5}, shows that
\[
    \int_{\mathbb{D}}\left|V_{\mu,1}(u)\right|^{\theta}(1-|u|^2)^{n+\alpha-1}{\rm d}A(u)<\infty.
\]

$\bullet$(ii) Assume that $V_{\mu}^{\xi}\in L_{\alpha}^{\theta}(\mathbb{B}_n)$. If $u\in E_k$ and $t\in [r_k,1)$, then $|1-tu|\leq (1-t)+t|1-u|\lesssim \delta_k$. Hence
\[
    V_{\mu}(u)\geq \int_{[r_k,1)}\frac{{\rm d}\mu(t)}{|1-tu|^{\beta}}\gtrsim \frac{\hat{\mu}(r_k)}{\delta_k^{\,\beta}}=h_k\cdot \delta_k^{-(n+1+\alpha)/\theta}.
\]
Then by Lemma \ref{lemma4.1} and Lemma \ref{lemma4.2}, we obtain
\begin{align*}
\int_{\mathbb B_n}\left|V_{\mu}^{\xi}(z)\right|^{\theta}{\rm d}v_{\alpha}(z)&\simeq\int_{\mathbb{D}}\left|V_{\mu}(u)\right|^{\theta}(1-|u|^2)^{n+\alpha-1}{\rm d}A(u)\\
&\geq \sum_{k=1}^{\infty}\int_{E_k}\left|V_{\mu}(u)\right|^{\theta}(1-|u|^2)^{n+\alpha-1}{\rm d}A(u)\\
&\gtrsim \sum_{k=1}^{\infty}\delta_k^{\,n+1+\alpha}\cdot \big(h_k\cdot \delta_k^{-(n+1+\alpha)/\theta}\big)^{\theta}=\sum_{k=1}^{\infty}h_k^{\,\theta}.
\end{align*}
Consequently, $\{h_k\}\in \ell^{\,\theta}$. The proof is now complete.
\end{proof}

We are now ready to prove the sufficiency of Theorem \ref{theorem1.2}.

\begin{proposition}\label{proposition4.4}
Let $\alpha>-1$, $\beta>0$ and   $0<q<p\leq\infty$. Set
\[
    \frac{1}{\theta}=\frac{1}q-\frac{1}p,
    \qquad
    \lambda=\beta+(n+1+\alpha)\Big(\frac{1}{p}-\frac{1}{q}\Big)
    =\beta-\frac{n+1+\alpha}{\theta},
\] 
 where $1/\infty=0$.  Assume $\mu$ is finite and
\[
    \int_{0}^{1}\left(\frac{\hat{\mu}(r)}{(1-r)^{\lambda}}\right)^{\theta}\frac{{\rm d}r}{1-r}<\infty.
\]
Then  $\mathcal{C}_{\mu,\beta}^{\xi}: A_{\alpha}^p(\mathbb{B}_n)\to A_{\alpha}^q(\mathbb{B}_n)$ is bounded. \end{proposition}

\begin{proof}
For $f\in A_\alpha^p(\mathbb B_n)$, put
$\mathcal Rf(z)=\sup_{0\leq s\leq1}|f(sz)|$. Positivity of $\mu$ gives the
pointwise estimate
\[
  |\mathcal C_{\mu,\beta}^{\xi}f(z)|
  \leq \mathcal Rf(z)V_\mu^\xi(z),
  \qquad z\in\mathbb B_n.
\]
Suppose first that $p<\infty$. By H\"older's inequality,
Lemma~\ref{lemma2.5}, and Lemma~\ref{lemma4.3}, we obtain
\[
\begin{split}
  \|\mathcal C_{\mu,\beta}^{\xi}f\|_{q,\alpha}
  \leq \|\mathcal Rf\|_{p,\alpha}
     \|V_\mu^\xi\|_{L_\alpha^\theta} \lesssim \|f\|_{p,\alpha}\|V_\mu^\xi\|_{L_\alpha^\theta}<\infty.
\end{split}
\]
If $p=\infty$, then $\theta=q$ and $\mathcal Rf\leq\|f\|_\infty$.
The same estimate therefore follows directly from
\[
  |\mathcal C_{\mu,\beta}^{\xi}f(z)|
  \leq\|f\|_\infty V_\mu^\xi(z)
\]
and Lemma~\ref{lemma4.3}. This completes the proof.
\end{proof}

We proceed to prove the necessity of Theorem \ref{theorem1.2}. Again some lemmas and basic facts are needed.

\begin{lemma}\label{lemma4.5}
Let   $0<p<\infty$  and $r_k=1-2^{-k}$, $k\geq 1$. Define
\begin{equation*}
f_k(z)=\frac{(1-r_k^2)^{(n+1+\alpha)/p}}{\big(1-r_k\langle z,\xi\rangle\big)^{2(n+1+\alpha)/p}},\quad z\in\mathbb{B}_n.
\end{equation*}
Then for any sequence $\{c_k\}\in \ell^p$, $\sum\limits_kc_kf_k\in A_{\alpha}^p(\mathbb{B}_n)$ and 
\[
    \left\|\sum_{k}c_kf_k\right\|_{p,\alpha}\lesssim \Big(\sum_{k}|c_k|^p\Big)^{1/p}.
\]
\end{lemma}

\begin{proof}
 Assume first that $1<p<\infty$.
 By an argument similar to that used in Lemma~\ref{lemma4.1}, we have
\begin{align}\label{equa4.6}
\Big\|\sum_{k}c_kf_k\Big\|_{p,\alpha}^p\simeq \int_{\mathbb{D}}\left|\sum_{k}c_k\frac{(1-r_k^2)^{(n+1+\alpha)/p}}{(1-r_k u)^{2(n+1+\alpha)/p}}\right|^p(1-|u|^2)^{n+\alpha-1}{\rm d}A(u).
\end{align}
Recall that $E_m=\{u\in\mathbb{D}: \delta_{m+1}<|1-u|\leq \delta_m\}$ and $\delta_m=2^{-m}$. On the set $E_m$, by \eqref{equa4.4}, we have
\begin{align*}
|1-r_k u|\gtrsim (1-r_k)+|1-u|\simeq \delta_k+\delta_m,\quad k\geq 1.
\end{align*}
Thus, for $u\in E_m$,
\begin{align*}
\frac{(1-r_k^2)^{(n+1+\alpha)/p}}{|1-r_k u|^{2(n+1+\alpha)/p}}\lesssim \frac{\delta_k^{\,(n+1+\alpha)/p}}{\big(\delta_k+\delta_m\big)^{2(n+1+\alpha)/p}}\lesssim \delta_m^{-(n+1+\alpha)/p}\cdot 2^{-|k-m|(n+1+\alpha)/p}.
\end{align*}
Set $d_m=\sum\limits_{k}|c_k|\cdot 2^{-|k-m|(n+1+\alpha)/p}$. 
Using Lemma \ref{lemma4.2}, we obtain
\begin{equation*}
\begin{split}
&\quad \int_{E_m}\left|\sum_{k}c_k\frac{(1-r_k^2)^{(n+1+\alpha)/p}}{(1-r_ku)^{2(n+1+\alpha)/p}}\right|^p(1-|u|^2)^{n+\alpha-1}{\rm d}A(u)\\
&\lesssim \int_{E_m}(1-|u|^2)^{n+\alpha-1}{\rm d}A(u)\cdot \delta_m^{\,-(n+1+\alpha)}\cdot d_m^{\,p}\lesssim d_m^{\,p}.
\end{split}
\end{equation*}
Since $\big\{2^{-|m|(n+1+\alpha)/p}\big\}\in \ell^1(\mathbb{Z})$, the discrete Young inequality gives $\sum_{m}d_m^{\,p}\lesssim \sum_{k}|c_k|^p$.
Hence 
\begin{equation}\label{equa4.7}
\begin{split}
&\quad \int_{\mathbb{D}\cap\big\{u: |1-u|<1/2\big\}}\left|\sum_{k}c_k\frac{(1-r_k^2)^{(n+1+\alpha)/p}}{(1-r_ku)^{2(n+1+\alpha)/p}}\right|^p(1-|u|^2)^{n+\alpha-1}{\rm d}A(u)\\
&=\sum_{m=1}^{\infty}\int_{E_m}\left|\sum_{k}c_k\frac{(1-r_k^2)^{(n+1+\alpha)/p}}{(1-r_ku)^{2(n+1+\alpha)/p}}\right|^p(1-|u|^2)^{n+\alpha-1}{\rm d}A(u)\\
&\lesssim \sum_{m=1}^{\infty}d_m^{\,p}\lesssim \sum_{k}|c_k|^p.
\end{split}
\end{equation}

On the other hand, if $u\in\mathbb{D}$ and $|1-u|\geq 1/2$, then $|1-r_k u|\gtrsim 1$. Hence, by H\"older's inequality, we obtain
\begin{equation}\label{equa4.8}
\begin{split}
&\quad \int_{\mathbb{D}\cap\big\{u: |1-u|\geq 1/2\big\}}\left|\sum_{k}c_k\frac{(1-r_k^2)^{(n+1+\alpha)/p}}{(1-r_ku)^{2(n+1+\alpha)/p}}\right|^p(1-|u|^2)^{n+\alpha-1}{\rm d}A(u)\\
&\lesssim \int_{\mathbb{D}}\left|\sum_{k}|c_k|\cdot \delta_k^{\,(n+1+\alpha)/p}\right|^p(1-|u|^2)^{n+\alpha-1}{\rm d}A(u)\\
&\lesssim \Big(\sum_{k}|c_k|^p\Big)\cdot \Big(\sum_{k}\delta_k^{\,(n+1+\alpha)/(p-1)}\Big)^{p-1}\lesssim \sum_{k}|c_k|^p.
\end{split}
\end{equation}
Therefore, combining \eqref{equa4.6}, \eqref{equa4.7} and \eqref{equa4.8}, we conclude that
\begin{align*}
\big\|\sum_{k}c_kf_k\big\|_{p,\alpha}^p\lesssim \sum_{k}|c_k|^p.
\end{align*}
 It remains to consider $0<p\leq1$. Lemma~\ref{lemma2.1} gives
\[
  \|f_k\|_{p,\alpha}^p
  \lesssim (1-r_k^2)^{n+1+\alpha}
  (1-r_k^2)^{-(n+1+\alpha)}\lesssim1.
\]
Hence, for $L>N$, subadditivity yields
\[
  \left\|\sum_{k=N+1}^Lc_kf_k\right\|_{p,\alpha}^p
  \leq\sum_{k=N+1}^L|c_k|^p\|f_k\|_{p,\alpha}^p
  \lesssim\sum_{k=N+1}^L|c_k|^p.
\]
Since $\{c_k\}\in\ell^p$,
\[
  \lim_{N\to\infty}\sup_{L>N}
  \left\|\sum_{k=N+1}^Lc_kf_k\right\|_{p,\alpha}^p=0.
\]
Taking $N=0$ and letting $L\to\infty$ gives the required estimate.  The proof is complete.
\end{proof}

\begin{lemma}\label{lemma4.6}
Write $B(z,\tau)$ for the Bergman metric ball centered at $z\in\mathbb{B}_n$ and with radius $\tau>0$. For $k\geq 1$, let $a_k=r_k\xi=(1-2^{-k})\xi$, then there exists $\tau_0>0$ such that $B(a_k,\tau_0)$ are pairwise disjoint.
\end{lemma}

\begin{proof}
For $a\in\mathbb{B}_n$, denote by $\varphi_a$ the involutive automorphism of $\mathbb{B}_n$ and write $\beta(z,w)$ for the Bergman metric between $z,w\in\mathbb{B}_n$. It is known that 
\[
    \beta(a_j,a_k)=\frac{1}{2}\log\frac{1+|\varphi_{a_j}(a_k)|}{1-|\varphi_{a_j}(a_k)|}.
\]
Assume $j>k\geq 1$, then
\[
    \big|\varphi_{a_j}(a_k)\big|=\left|\frac{a_j-a_k}{1-\langle a_j,a_k\rangle}\right|=\frac{r_j-r_k}{1-r_jr_k}=\frac{2^{-k}-2^{-j}}{2^{-k}+2^{-j}-2^{-k-j}}\geq \frac{1}{3}.
\]
Hence taking $\tau_0<\frac{1}{4}\log 2$, then $B(a_j,\tau_0)\cap B(a_k,\tau_0)=\emptyset$ whenever $j\neq k$.
\end{proof}

Let $\{\rho_j(x)\}$ denote the sequence of Rademacher functions, defined by 
\begin{equation*}
\rho_0(x)=\left\{
\begin{aligned}
&1, &\text{if}~~ 0\leq x-[x]<\frac{1}{2},\\
&-1, &\text{if}~~ \frac{1}{2}\leq x-[x]<1,
\end{aligned}
\right.
\end{equation*}
where $[x]$ denotes the largest integer not greater than $x$ and $\rho_{j}(x)=\rho_0(2^jx)$ for $j\geq 1$. If $0<p<\infty$, then Khinchine's inequality states that
\begin{equation*}
\left(\sum_{j}|c_j|^2\right)^{p/2}\simeq \int_{0}^1\left|\sum_{j}c_j\rho_j(x)\right|^p{\rm d}x
\end{equation*}
for any complex sequence $\{c_j\}$.

We are now ready to prove the necessity of Theorem \ref{theorem1.2}.

\begin{proposition}\label{proposition4.7}
Let $\alpha>-1$, $\beta>0$ and  \;  $0<q<p<\infty$. Set
\[
    \theta=\frac{pq}{p-q},\quad \lambda=\beta+(n+1+\alpha)\Big(\frac{1}{p}-\frac{1}{q}\Big).
\]
If $\mathcal{C}_{\mu,\beta}^{\xi}: A_{\alpha}^p(\mathbb{B}_n)\to A_{\alpha}^q(\mathbb{B}_n)$ is bounded, then $\mu$ is finite and 
\[
    \int_{0}^1\left(\frac{\hat{\mu}(r)}{(1-r)^{\lambda}}\right)^{\theta}\frac{{\rm d}r}{1-r}<\infty.
\]
\end{proposition}

\begin{proof}
Assume $\mathcal{C}_{\mu,\beta}^{\xi}: A_{\alpha}^p(\mathbb{B}_n)\to A_{\alpha}^q(\mathbb{B}_n)$ is bounded, then by \cite[Theorem 2.1]{Zhu2005},
\[
    \mu([0,1))=\big(\mathcal{C}_{\mu,\beta}^{\xi}1\big)(0)\lesssim 1.
\]
For $k\geq 1$, recall that $r_k=1-\delta_k$ and $\delta_k=2^{-k}$. For any nonnegative sequence $\mathbf{c}=\{c_k\}\in \ell^p$, let
\[
    F_{\mathbf{c},p}(z)=\sum_{k=1}^{\infty}c_kf_k(z)=\sum_{k=1}^{\infty}c_k\frac{(1-r_k^2)^{(n+1+\alpha)/p}}{(1-r_k\langle z,\xi\rangle)^{2(n+1+\alpha)/p}},\quad z\in\mathbb{B}_n.
\]
By Lemma \ref{lemma4.5}, the boundedness of $\mathcal{C}_{\mu,\beta}^{\xi}$ implies that
\begin{equation*}
\big\|\mathcal{C}_{\mu,\beta}^{\xi}F_{\mathbf{c},p}\big\|_{q,\alpha}^q=\int_{\mathbb{B}_n}\left|\sum_{k}c_k\int_{[0,1)}\frac{f_k(tz)}{(1-t\langle z,\xi\rangle)^{\beta}}{\rm d}\mu(t)\right|^q{\rm d}v_{\alpha}(z)\lesssim \big\|\{c_k\}\big\|_{\ell^p}^q.
\end{equation*}
Replacing $c_k$ by $c_k\rho_k(x)$ in the above inequality and integrating both sides with respect to $x$ over the interval $[0,1]$, we obtain
\begin{equation*}
\int_{0}^1\int_{\mathbb{B}_n}\left|\sum_{k}c_k\rho_k(x)\int_{[0,1)}\frac{f_k(tz)}{(1-t\langle z,\xi\rangle)^{\beta}}{\rm d}\mu(t)\right|^q{\rm d}v_{\alpha}(z)\,{\rm d}x\lesssim \big\|\{c_k\}\big\|_{\ell^p}^q.
\end{equation*}
Then it follows from Khinchine's inequality and Fubini's theorem that
\begin{equation}\label{equa4.9}
\begin{split}
&\quad \int_{\mathbb{B}_n}\left(\sum_{k}|c_k|^2\Big|\int_{[0,1)}\frac{f_k(tz)}{(1-t\langle z,\xi\rangle)^{\beta}}{\rm d}\mu(t)\Big|^2\right)^{q/2}{\rm d}v_{\alpha}(z)\\
&\lesssim \int_{\mathbb{B}_n}\int_{0}^1\left|\sum_{k}c_k\rho_k(x)\int_{[0,1)}\frac{f_k(tz)}{(1-t\langle z,\xi\rangle)^{\beta}}{\rm d}\mu(t)\right|^q{\rm d}x\,{\rm d}v_{\alpha}(z)\\
&=\int_{0}^1\int_{\mathbb{B}_n}\left|\sum_{k}c_k\rho_k(x)\int_{[0,1)}\frac{f_k(tz)}{(1-t\langle z,\xi\rangle)^{\beta}}{\rm d}\mu(t)\right|^q{\rm d}v_{\alpha}(z)\,{\rm d}x\\
&\lesssim \|\{c_j\}\|_{\ell^p}^q.
\end{split}
\end{equation}
On the other hand, if $t\in [r_j,1)$, then
\[
    1-r_j^2 t\simeq 1-r_j t\simeq 1-r_j^2\simeq \delta_j.
\]
So we use Lemma \ref{lemma4.6} and \cite[Theorem 2.1]{Zhu2005} to obtain
\begin{equation}\label{equa4.10}
\begin{split}
&\quad \int_{\mathbb{B}_n}\left(\sum_{k}|c_k|^2\left|\int_{[0,1)}\frac{f_k(tz)}{(1-t\langle z,\xi\rangle)^{\beta}}{\rm d}\mu(t)\right|^2\right)^{q/2}{\rm d}v_{\alpha}(z)\\
&\geq \sum_{j=1}^{\infty}\int_{B(r_j\xi,\tau_0)}|c_j|^q\left|\int_{[0,1)}\frac{f_j(tz)}{(1-t\langle z,\xi\rangle)^{\beta}}{\rm d}\mu(t)\right|^q{\rm d}v_{\alpha}(z)\\
&\gtrsim \sum_{j=1}^{\infty}|c_j|^q(1-r_j^2)^{n+1+\alpha}\left(\int_{[0,1)}\frac{f_j(t r_j\xi)}{(1-t r_j)^{\beta}}{\rm d}\mu(t)\right)^q\\
&\gtrsim \sum_{j=1}^{\infty}|c_j|^q(1-r_j^2)^{n+1+\alpha}\left(\int_{[r_j,1)}\frac{(1-r_j^2)^{(n+1+\alpha)/p}}{(1-t r_j^2)^{2(n+1+\alpha)/p}(1-t r_j)^{\beta}}{\rm d}\mu(t)\right)^q\\
&\simeq \sum_{j=1}^{\infty}|c_j|^q\left(\frac{\hat{\mu}(r_j)}{\delta_j^{\,\lambda}}\right)^q.
\end{split}
\end{equation}
Therefore, combining \eqref{equa4.9} and \eqref{equa4.10}, we get
\begin{equation*}
\sum_{j=1}^{\infty}|c_j|^q\left(\frac{\hat{\mu}(r_j)}{\delta_j^{\lambda}}\right)^q\lesssim \|\{c_j\}\|_{\ell^p}^q.
\end{equation*}
This means the sequence $\big\{\big(\frac{\hat{\mu}(r_j)}{\delta_j^{\lambda}}\big)^q\big\}$ belongs to the dual of $\ell^{\frac{p}{q}}$, or equivalently,
\begin{equation*}
\sum_{j=1}^{\infty}\left(\frac{\hat{\mu}(r_j)}{\delta_j^{\lambda}}\right)^{\frac{pq}{p-q}}<\infty.
\end{equation*}
By Lemma \ref{lemma2.7}, this is then equivalent to
\begin{equation*}
\int_{0}^1\left(\frac{\hat{\mu}(r)}{(1-r)^{\lambda}}\right)^{\theta}\frac{{\rm d}r}{1-r}<\infty.
\end{equation*}
The proof is now complete.
\end{proof}

{
\begin{proposition}\label{proposition4.8}
Let $\alpha>-1$, $\beta>0$ and $0<q<\infty$, and set
\[
    \lambda=\beta-\frac{n+1+\alpha}{q}.
\]
If $\mathcal{C}_{\mu,\beta}^{\xi}:H^{\infty}(\mathbb{B}_n)\to A_{\alpha}^q(\mathbb{B}_n)$ is bounded, then $\mu$ is finite and
\[
    \left\{\frac{\hat{\mu}(r_k)}{\delta_k^{\lambda}}\right\}_{k=1}^{\infty}\in\ell^q.
\]
\end{proposition}
}

\begin{proof}
Suppose that $\mathcal{C}_{\mu,\beta}^{\xi}:H^{\infty}(\mathbb{B}_n)\to A_{\alpha}^q(\mathbb{B}_n)$ is bounded. Taking $f\equiv1$, we obtain
$F:=\mathcal{C}_{\mu,\beta}^{\xi}1\in A_{\alpha}^q(\mathbb{B}_n)$ and
\[
    \mu([0,1))=F(0)<\infty.
\]
After a unitary change of variables, we may assume that $\xi=e_1$ and write $F(z)=G(z_1)$, where
\[
    G(u)=\int_{[0,1)}(1-tu)^{-\beta}{\rm d}\mu(t),
    \qquad u\in\mathbb{D}.
\]
As in the proof of Lemma \ref{lemma4.1},
\[
    \|F\|_{q,\alpha}^{q}\simeq
    \int_{\mathbb{D}}|G(u)|^q
    (1-|u|^2)^{n+\alpha-1}{\rm d}A(u).
\]

Choose $0<\varepsilon<\min\{\pi/3,\pi/(2\beta)\}$ and, for $k\geq1$, set
\[
    S_k=\big\{u\in\mathbb{D}:\delta_{k+1}<|1-u|\leq\delta_k,
    \ |\arg(1-u)|<\varepsilon\big\}.
\]
Since $1-tu=(1-t)+t(1-u)$, if $u\in S_k$, then
\[
    |{\rm arg}(1-tu)|<\varepsilon,
    \qquad
    |{\rm arg}(1-tu)^{-\beta}|\leq\beta\varepsilon<\frac{\pi}{2}.
\]
Hence
\[
    {\rm Re}(1-tu)^{-\beta}
    \geq \cos(\beta\varepsilon)|1-tu|^{-\beta}\geq0.
\]
When $t\in[r_k,1)$ and $u\in S_k$, we have $|1-tu|\leq2\delta_k$. Therefore
\begin{align*}
|G(u)|\geq{\rm Re}G(u)
&=\int_{[0,1)}{\rm Re}(1-tu)^{-\beta}{\rm d}\mu(t)\\
&\gtrsim\int_{[r_k,1)}{\rm Re}(1-tu)^{-\beta}{\rm d}\mu(t)\\
&\gtrsim\delta_k^{-\beta}\hat{\mu}(r_k).
\end{align*}
Moreover, the same computation as in Lemma \ref{lemma4.2} gives
\[
    \int_{S_k}(1-|u|^2)^{n+\alpha-1}{\rm d}A(u)
    \simeq\delta_k^{n+1+\alpha}.
\]
Since the sets $S_k$ are pairwise disjoint,
\begin{align*}
\|F\|_{q,\alpha}^{q}
&\gtrsim\sum_{k\geq1}\int_{S_k}|G(u)|^q
(1-|u|^2)^{n+\alpha-1}{\rm d}A(u)\\
&\gtrsim\sum_{k\geq1}
\left(
\frac{\hat{\mu}(r_k)}
{\delta_k^{\,\beta-(n+1+\alpha)/q}}
\right)^q
=
\sum_{k\geq1}
\left(\frac{\hat{\mu}(r_k)}{\delta_k^{\lambda}}\right)^q.
\end{align*}
Thus the asserted $\ell^q$ condition follows.
\end{proof}

\begin{proof}[{\bf Proof of Theorem 1.2}]
The sufficiency follows from Proposition \ref{proposition4.4}, and the necessity follows from Proposition \ref{proposition4.7}\ {for $p<\infty$, while for $p=\infty$ it follows from Proposition \ref{proposition4.8}}.
\end{proof}

\section{Proof of Theorem 1.3}

We first consider the boundedness of $\mathcal{C}_{\mu,\beta}^{\xi}$ on
$H^\infty(\mathbb{B}_n)$. The following result establishes assertion~{(iii)} of
Theorem~\ref{theorem1.3}.

\begin{theorem}\label{theorem5.1}
Let $\alpha>-1$, $\beta>0$ and $\mu$ be a positive Borel measure on $[0,1)$. Then $\mathcal C_{\mu,\beta}^{\xi}: H^{\infty}(\mathbb{B}_n)\to H^{\infty}(\mathbb{B}_n)$ is bounded if and only if 
\begin{equation}\label{equa5.1}
    \int_{[0,1)}\frac{{\rm d}\mu(t)}{(1-t)^{\beta}}<\infty.
\end{equation}
\end{theorem}

\begin{proof}
Assume first that the integral condition in \eqref{equa5.1} holds. Then for any $f\in H^{\infty}(\mathbb{B}_n)$ and $z\in\mathbb{B}_n$, we have
\begin{align*}
\left|\big(\mathcal{C}_{\mu,\beta}^{\xi}f\big)(z)\right|\leq \|f\|_{\infty}\cdot \int_{[0,1)}\frac{{\rm d}\mu(t)}{|1-t\langle z,\xi\rangle|^{\beta}}\leq \|f\|_{\infty}\cdot \int_{[0,1)}\frac{{\rm d}\mu(t)}{(1-t)^{\beta}}.
\end{align*}
Thus $\mathcal{C}_{\mu,\beta}^{\xi}$ maps $H^{\infty}(\mathbb{B}_n)$ boundedly into $H^{\infty}(\mathbb{B}_n)$.

Conversely, suppose $\mathcal{C}_{\mu,\beta}^{\xi}: H^{\infty}(\mathbb{B}_n)\to H^{\infty}(\mathbb{B}_n)$ is bounded. Taking $f\equiv 1$, we get
\[
    \big(\mathcal{C}_{\mu,\beta}^{\xi}1\big)(r\xi)=\int_{[0,1)}\frac{{\rm d}\mu(t)}{(1-rt)^{\beta}}.
\]
The left hand side is bounded uniformly in $r$. Then the monotone convergence theorem gives
\[
    \int_{[0,1)}\frac{{\rm d}\mu(t)}{(1-t)^{\beta}}<\infty.\qedhere
\]
\end{proof}

We next treat the endpoint case $p=1$ in assertion~{(i)} of
Theorem~\ref{theorem1.3}. The corresponding characterization is as follows.

\begin{theorem}\label{theorem5.2}
Let $\alpha>-1$, $\beta>0$ and $\mu$ be a positive Borel measure on $[0,1)$. Then $\mathcal C_{\mu,\beta}^{\xi}: A_{\alpha}^1(\mathbb{B}_n)\to H^{\infty}(\mathbb{B}_n)$ is bounded if and only if $\mu$ is a $(\beta+n+1+\alpha)$-Carleson measure on $[0,1)$, that is
\begin{equation}\label{equa5.0002}
    \sup_{r\in [0,1)}\frac{\hat{\mu}(r)}{(1-r)^{\beta+n+1+\alpha}}<\infty.
\end{equation}
\end{theorem}

\begin{proof}
Assume first that $\mu$ is a $(\beta+n+1+\alpha)$-Carleson measure on $[0,1)$. For any $\delta>0$ and $z\in\mathbb{B}_n$, by Lemma \ref{lemma2.6}, we have
\begin{align*}
&\quad\sup_{w\in \mathbb{B}_n}\int_{[0,1)}\frac{(1-|w|)^{\delta}}{\left|1-t\langle z,\xi\rangle\right|^{\beta}|1-t\langle z,w\rangle|^{n+1+\alpha+\delta}}{\rm d}\mu(t)\\
&\lesssim \sup_{w\in\mathbb{B}_n}\int_{[0,1)}\frac{(1-|w|)^{\delta}}{(1-t)^{\beta}(1-t|w|)^{n+1+\alpha+\delta}}{\rm d}\mu(t)<\infty.
\end{align*}
Then for any $f\in A_{\alpha}^1(\mathbb{B}_n)$, we get 
\begin{align*}
\big|\big(\mathcal{C}_{\mu,\beta}^{\xi}f\big)(z)\big|&\leq \left(\mathcal{B}_{\mu,\beta,\delta}^{\xi}|f|\right)(z)\\
&\simeq \int_{\mathbb{B}_n}|f(w)|\left(\int_{[0,1)}\frac{(1-|w|^2)^{\delta}}{|1-t\langle z,\xi\rangle|^{\beta}|1-t\langle z,w\rangle|^{n+1+\alpha+\delta}}{\rm d}\mu(t)\right){\rm d}v_{\alpha}(w)\\
&\lesssim \int_{\mathbb{B}_n}|f(w)|{\rm d}v_{\alpha}(w)
\end{align*}
uniformly in $z\in\mathbb{B}_n$. Consequently, $\mathcal C_{\mu,\beta}^{\xi}$  is bounded from $A_{\alpha}^1(\mathbb{B}_n)$ to $H^{\infty}(\mathbb{B}_n)$.

Conversely, assume that $\mathcal{C}_{\mu,\beta}^{\xi}: A_{\alpha}^1(\mathbb{B}_n)\to H^{\infty}(\mathbb{B}_n)$ is bounded. For any $r\in [0,1)$, define
\[
    f_r(z)=\frac{(1-r^2)^{n+1+\alpha}}{(1-r\langle z,\xi\rangle)^{2(n+1+\alpha)}},\quad z\in\mathbb{B}_n.
\]
By Lemma \ref{lemma2.1}, $f_r\in A_{\alpha}^1(\mathbb{B}_n)$ and $\|f_r\|_{1,\alpha}\simeq 1$. Then
\[
    \sup_{r\in [0,1)}\left|\mathcal{C}_{\mu,\beta}^{\xi}f_r(r\xi)\right|<\infty.
\]
On the other hand,
\begin{align*}
\mathcal{C}_{\mu,\beta}^{\xi}f_r(r\xi)&=\int_{[0,1)}\frac{(1-r^2)^{n+1+\alpha}}{(1-r^2t)^{2(n+1+\alpha)}(1-tr)^{\beta}}{\rm d}\mu(t)\\
&\geq \int_{[r,1)}\frac{(1-r^2)^{n+1+\alpha-\beta}}{(1-r^3)^{2(n+1+\alpha)}}{\rm d}\mu(t)\simeq \frac{\hat{\mu}(r)}{(1-r)^{n+1+\alpha+\beta}}.
\end{align*}
Therefore,
\[
    \sup_{r\in [0,1)}\frac{\hat{\mu}(r)}{(1-r)^{n+1+\alpha+\beta}}<\infty,
\]
which means that $\mu$ is an $(n+1+\alpha+\beta)$-Carleson measure on $[0,1)$.
\end{proof}

 \begin{proof}[Proof of Theorem~\ref{theorem1.3}(i) for $0<p<1$]
Set
\[
  \lambda=\beta+\frac{n+1+\alpha}{p}.
\]
If $\mu$ is a $\lambda$-Carleson measure, then Proposition~\ref{proposition3.5} gives
\[
  \|\mathcal C_{\mu,\beta}^{\xi}f\|_\infty
  \lesssim
  \left(\sup_{0\leq s<1}\frac{\hat\mu(s)}{(1-s)^\lambda}\right)
  \|f\|_{p,\alpha},
  \qquad f\in A_\alpha^p(\mathbb B_n).
\]
This proves sufficiency.

Conversely, suppose that
$\mathcal C_{\mu,\beta}^{\xi}:A_\alpha^p(\mathbb B_n)\to
H^\infty(\mathbb B_n)$ is bounded. Taking $f\equiv1$ and evaluating at the
origin first shows that $\mu([0,1))<\infty$. For $0\leq r<1$, define the
normalized peak function
\[
  f_{r,p}(z)=\frac{(1-r^2)^{(n+1+\alpha)/p}}{(1-r\langle z,\xi\rangle)^{2(n+1+\alpha)/p}}.
\]
By Lemma~\ref{lemma2.1}, $\|f_{r,p}\|_{p,\alpha}\simeq1$, uniformly in
$r$. Since $\mathcal C_{\mu,\beta}^{\xi}$ is bounded,

\[
  \sup_{0\leq r<1}
  \left|\big(\mathcal C_{\mu,\beta}^{\xi}f_{r,p}\big)(r\xi)\right|<\infty.
\]
On the other hand, positivity of the integrand on the radius through $\xi$ gives, for each $0\leq r<1$,

\[
{
\begin{aligned}
  \big(\mathcal C_{\mu,\beta}^{\xi}f_{r,p}\big)(r\xi)
  &=\int_{[0,1)}
  \frac{(1-r^2)^{(n+1+\alpha)/p}}{(1-tr^2)^{2(n+1+\alpha)/p}(1-tr)^\beta}
  {\rm d}\mu(t)\\
  &\geq\int_{[r,1)}
  \frac{(1-r^2)^{(n+1+\alpha)/p}}{(1-tr^2)^{2(n+1+\alpha)/p}(1-tr)^\beta}
  {\rm d}\mu(t).
\end{aligned}}
\]
For $t\in[r,1)$, we have
$1-tr^2\leq1-r^3\simeq1-r$, $1-tr\leq1-r^2\simeq1-r$, and
$1-r^2\simeq1-r$. Consequently,

\[
  \big(\mathcal C_{\mu,\beta}^{\xi}f_{r,p}\big)(r\xi)
  \gtrsim
  \frac{\hat\mu(r)}{(1-r)^{\beta+(n+1+\alpha)/p}}.
\]
Taking the supremum over $r$ proves necessity.
Together with Theorem~\ref{theorem5.2}, this establishes
Theorem~\ref{theorem1.3}(i) for the full range $0<p\leq1$.
\end{proof}

 Now we consider the case $\mathcal{C}_{\mu,\beta}^{\xi}: 
 A_{\alpha}^p(\mathbb{B}_n)\to H^{\infty}(\mathbb{B}_n)$ for $1<p<\infty$. To this end, we need the following lemma.

\begin{lemma}\label{lemma5.3}
Let $0<p<\infty$ and $I_k=[1-2^{-k}, 1-2^{-(k+1)})$ for $k\geq 1$. Then
\begin{equation}\label{equa5.2}
\sum_{k=1}^{\infty}2^{-k(n+1+\alpha)}\sup_{t\in I_k}|f(tz)|^p\lesssim \|f\|_{p,\alpha}^p
\end{equation}
for all $f\in A_{\alpha}^p(\mathbb{B}_n)$ and $z\in\mathbb{B}_n$.
\end{lemma}

\begin{proof}
When $|z|\leq 1/2$, by the standard pointwise estimate \cite[Theorem 2.1]{Zhu2005}, we get
\begin{equation}\label{equa5.03}
    \sup_{t\in I_k}|f(tz)|^p\lesssim \frac{1}{(1-t^2|z|^2)^{n+1+\alpha}}\|f\|_{p,\alpha}^p\lesssim \|f\|_{p,\alpha}^p,\quad \forall k\geq 1.
\end{equation}
Then \eqref{equa5.2} holds by the convergence of $\sum_{k=1}^{\infty}2^{-k(n+1+\alpha)}$.

Now we assume that $|z|>1/2$. For $k\geq 1$, denote by $r_k=1-2^{-k}$. Choose $N\geq 1$ such that
\[
    r_N\leq |z|<r_{N+1}.
\]
Set $z_k=r_k z$. For $k\leq N$ and $t\in I_k$, we have
$2^{-k}\leq 1-|z_k|\leq 2^{-(k-1)}$, and
\[
    \left|\frac{z_k-tz}{1-\langle z_k, tz\rangle}\right|\leq \frac{|t-r_k|}{1-tr_k|z|^2}\leq \frac{|t-r_k|}{1-r_kt}\leq \frac{1}{2}.
\]
Then there exists $\tau, \tau'>0$ such that $B(tz,\tau)\subset B(z_k,\tau')$. Using \cite[Theorem 2.1]{Zhu2005} again, we get
\begin{align*}
2^{-k(n+1+\alpha)}\sup_{t\in I_k}|f(tz)|^p&\lesssim \frac{2^{-k(n+1+\alpha)}}{(1-t^2|z|^2)^{n+1+\alpha}}\int_{D(tz,\tau)}|f(w)|^p{\rm d}v_{\alpha}(w)\\
&\lesssim \int_{B(z_k,\tau')}|f(w)|^p{\rm d}v_{\alpha}(w).
\end{align*}
If $k\neq j$ and $w\in B(z_k,\tau')\cap B(z_j,\tau')$, then \cite[Lemma 2.20]{Zhu2005} gives
\[
    1-|w|^2\simeq 1-|z_k|^2\simeq 1-|z_j|^2,
\]
which implies that $2^{-k}\simeq 2^{-j}$. Hence $|k-j|\leq C$ for some $C>0$. It follows that
\[
\sum_{k=1}^{\infty}\chi_{B(z_k,\tau')}\lesssim 1.
\]
Therefore, when $r_N\leq |z|<r_{N+1}$, we obtain
\begin{equation}\label{equa5.3}
\sum_{k=1}^{N}2^{-k(n+1+\alpha)}\sup_{t\in I_k}|f(tz)|^p\lesssim \sum_{k=1}^{\infty}\int_{B(z_k,\tau')}|f(w)|^p{\rm d}v_{\alpha}(w)
\lesssim \|f\|_{p,\alpha}^{p}.
\end{equation}
On the other hand, for $k\geq N+1$ and $t\in I_k$, we have
\begin{equation*}
\left|\frac{tz-z}{1-\langle tz,z\rangle}\right|\leq \frac{|z|}{1+|z|}\frac{1-t}{1-|z|}\leq \frac{2^{-k}}{2^{-(N+1)}}\frac{|z|}{1+|z|}\leq \frac{1}{2}.
\end{equation*}
Moreover,
\[
    \sum_{k\geq N+1}2^{-k(n+1+\alpha)}\lesssim (1-|z|)^{n+1+\alpha}.
\]
Hence
\begin{equation}\label{equa5.4}
\begin{split}
&\quad \sum_{k\geq N+1}2^{-k(n+1+\alpha)}\sup_{t\in I_k}|f(tz)|^p\\
&\lesssim \sum_{k\geq N+1}2^{-k(n+1+\alpha)}\sup_{t\in I_k}\frac{1}{(1-t^2 |z|^2)^{n+1+\alpha}}\int_{B(tz,\tau)}|f(w)|^p{\rm d}v_{\alpha}(w)\\
&\lesssim \frac{\sum_{k\geq N+1}2^{-k(n+1+\alpha)}}{(1-|z|^2)^{n+1+\alpha}}\int_{B(z,\tau')}|f(w)|^p{\rm d}v_{\alpha}(w)\lesssim \|f\|_{p,\alpha}^p.
\end{split}
\end{equation}
Combining \eqref{equa5.3} and \eqref{equa5.4}, we obtain
\begin{equation}\label{equa5.6}
\sum_{k=1}^{\infty}2^{-k(n+1+\alpha)}\sup_{t\in I_k}|f(tz)|^p\lesssim \|f\|_{p,\alpha}^p,\quad |z|>\frac{1}{2}.
\end{equation}
The proof is now complete by combining \eqref{equa5.03} and \eqref{equa5.6}.
\end{proof}

\begin{theorem}\label{theorem5.4}
Let $\alpha>-1$, $\beta>0$ and $1<p<\infty$.  Then $\mathcal{C}_{\mu,\beta}^{\xi}: A_{\alpha}^p(\mathbb{B}_n)\to H^{\infty}(\mathbb{B}_n)$ is bounded if and only if $\mu$ is finite and 
\begin{equation}\label{equa5.7}
    \left\{\frac{\hat{\mu}(r_k)}{\delta_k^{\,\beta+(n+1+\alpha)/p}}\right\}_{k=1}^{\infty}\in \ell^{p'},
\end{equation}
where $1/p+1/p'=1$.
\end{theorem}

\begin{proof}
Assume first that $\mu$ is finite and \eqref{equa5.7} holds. For $k\geq 1$, denote by $I_k=[r_k,r_{k+1})$ and $\mu_k=\mu(I_k)$. Then the assumption gives 
\[
    \left\{\frac{\mu_k}{\delta_k^{\beta+(n+1+\alpha)/p}}\right\}\in \ell^{p'}.
\] 
For any $z\in\mathbb{B}_n$ and $f\in A_{\alpha}^p(\mathbb{B}_n)$, we have
\begin{align*}
\big|\mathcal{C}_{\mu,\beta}^{\xi}f(z)\big|&\leq \int_{[0,1)}\frac{|f(tz)|}{(1-t)^{\beta}}{\rm d}\mu(t)\\
&\lesssim \sum_{k=1}^{\infty}\mu_k\delta_k^{-\beta}\cdot\sup_{t\in I_k}|f(tz)|+\mu([0,1/2))\cdot\sup_{t\in [0,1/2]}|f(tz)|.
\end{align*}
By the standard pointwise estimate in $A_{\alpha}^p(\mathbb{B}_n)$ (\cite[Theorem 2.1]{Zhu2005}), we get
\begin{equation}\label{equa5.8}
\sup_{t\in [0,1/2]}|f(tz)|
\lesssim \sup_{t\in[0,1/2]}
\frac{\|f\|_{p,\alpha}}{(1-t^2|z|^2)^{(n+1+\alpha)/p}}
\lesssim \|f\|_{p,\alpha},\quad z\in \mathbb{B}_n.
\end{equation}
By H\"older's inequality and Lemma \ref{lemma5.3}, we get
\begin{equation}\label{equa5.9}
\begin{split}
&\quad\sum_{k=1}^{\infty}\mu_k\delta_k^{-\beta}\cdot\sup_{t\in I_k}|f(tz)|\\
&\lesssim \left(\sum_{k=1}^{\infty}\Big(\frac{\mu_k}{\delta_k^{\beta+(n+1+\alpha)/p}}\Big)^{p'}\right)^{1/p'}\cdot\left(\sum_{k=1}^{\infty}\delta_k^{n+1+\alpha}\sup_{t\in I_k}|f(tz)|^p\right)^{1/p}\lesssim \|f\|_{p,\alpha}.
\end{split}
\end{equation}
Therefore, combining \eqref{equa5.8} and \eqref{equa5.9}, we obtain
\begin{align*}
\|\mathcal{C}_{\mu,\,\beta}^{\xi}f\|_{\infty}\lesssim \|f\|_{p,\alpha},\quad \forall f\in A_{\alpha}^p(\mathbb{B}_n).
\end{align*}
This shows the boundedness  of $\mathcal{C}_{\mu,\,\beta}^{\xi}$ from $A_{\alpha}^p(\mathbb{B}_n)$ to $H^{\infty}(\mathbb{B}_n)$.

Conversely, assume that $\mathcal{C}_{\mu,\,\beta}^{\xi}: A_{\alpha}^p(\mathbb{B}_n)\to H^{\infty}(\mathbb{B}_n)$ is bounded. Taking $f\equiv 1$, we get
\[
    \mu([0,1))=\big(\mathcal C_{\mu,\,\beta}^{\xi}1\big)(0)\lesssim 1.
\] 
Recall that the function 
\[
    f_k(z)=\frac{(1-r_k^2)^{(n+1+\alpha)/p}}{(1-r_k\langle z,\xi\rangle)^{2(n+1+\alpha)/p}},\quad k\geq 1,
\] 
and for any sequence $\mathbf{c}=\{c_k\}\in \ell^p$,
\[
    F_{\mathbf{c},p}(z)=\sum_{k=1}^{\infty}c_kf_k(z).
\]
By Lemma \ref{lemma4.5}, the boundedness of $\mathcal{C}_{\mu,\beta}^{\xi}$ implies that
\begin{equation*}
\sup_{z\in\mathbb{B}_n}\big|\big(\mathcal{C}_{\mu,\beta}^{\xi}F_{\mathbf{c},p}\big)(z)\big|\lesssim \|\{c_k\}\|_{\ell^p}.
\end{equation*} 
It follows that
\begin{align*}
&\quad \sup_{r\in (0,1)}\Bigg|\sum_{k=1}^{\infty}c_k\big(\mathcal{C}_{\mu,\beta}^{\xi}f_k\big)(r\xi)\Bigg|\\
&=\sup_{r\in (0,1)}\Bigg|\sum_{k=1}^{\infty}c_k\int_{[0,1)}\frac{(1-r_k^2)^{(n+1+\alpha)/p}}{(1-tr_kr)^{2(n+1+\alpha)/p}(1-tr)^{\beta}}{\rm d}\mu(t)\Bigg|\\
&\lesssim \|\{c_k\}\|_{\ell^p}.
\end{align*}
Hence, by $\ell^p$--$\ell^{p'}$ duality, for every $r\in (0,1)$ the sequence
\[
   \left\{\int_{[0,1)}\frac{(1-r_k^2)^{(n+1+\alpha)/p}}{(1-tr_kr)^{2(n+1+\alpha)/p}(1-tr)^{\beta}}{\rm d}\mu(t)\right\}_{k=1}^{\infty}
   \in \ell^{p'},
\]

\par\smallskip

{and its $\ell^{p'}$ norm is bounded uniformly in $r$.
For each $k$, the displayed coefficient increases as $r\uparrow1$. By the
monotone convergence theorem it tends to
\[
  A_k:=\int_{[0,1)}
  \frac{(1-r_k^2)^{(n+1+\alpha)/p}}{(1-tr_k)^{2(n+1+\alpha)/p}(1-t)^{\beta}}{\rm d}\mu(t).
\]
}

\par\smallskip

{Fatou's lemma for the counting measure now gives
\[
  \sum_{k=1}^{\infty}A_k^{p'}
  \leq\liminf_{r\uparrow1}\sum_{k=1}^{\infty}
  \left(\int_{[0,1)}\frac{(1-r_k^2)^{(n+1+\alpha)/p}}{(1-tr_kr)^{2(n+1+\alpha)/p}(1-tr)^{\beta}}{\rm d}\mu(t)\right)^{p'}
  <\infty.
\]
Finally, for $t\in[r_k,1)$,
$1-tr_k\lesssim\delta_k$, $1-t\leq\delta_k$, and
$1-r_k^2\simeq\delta_k$. Therefore
\[
  A_k\geq\int_{[r_k,1)}
  \frac{(1-r_k^2)^{(n+1+\alpha)/p}}{(1-tr_k)^{2(n+1+\alpha)/p}(1-t)^{\beta}}{\rm d}\mu(t)
  \gtrsim \frac{\hat\mu(r_k)}{\delta_k^{\beta+(n+1+\alpha)/p}}.
\]
Thus \eqref{equa5.7} holds.}

The proof is now complete.
\end{proof}

\begin{proof}[{\bf Proof of Theorem 1.3}]
{Item (i) follows from Theorem \ref{theorem5.2} and the proof for
$0<p<1$ given immediately after it. Item (ii) follows from Theorem
\ref{theorem5.4}, and item (iii) follows from Theorem \ref{theorem5.1}.}
\end{proof}

\begin{remark}
The characterization of the boundedness of $\mathcal{C}_{\mu,\beta}^{\xi}$ between weighted Bergman spaces and $H^{\infty}(\mathbb{B}_n)$ can be summarized in a unified form. In fact, by the standard dyadic decomposition of $[0,1)$ and the geometric Hardy inequality, condition \eqref{equa5.1} is equivalent to 
\[
    \mu([0,1))<\infty
    \quad\text{and}\quad
    \left\{\delta_k^{-\beta}\hat{\mu}(r_k)\right\}_{k=1}^{\infty}\in \ell^1.
\]
  More generally, for $0<p\leq1$, the condition in Theorem
\ref{theorem1.3}(i)  is equivalent to
\[
    \mu([0,1))<\infty
    \quad\text{and}\quad
    \left\{\frac{\hat{\mu}(r_k)}{\delta_k^{\beta+(n+1+\alpha)/p}}\right\}_{k=1}^{\infty}\in \ell^{\infty}.
\]
  At  $p=1$  this is exactly \eqref{equa5.0002} .

{Finally, when $p=\infty$, Theorem \ref{theorem1.2} has
$\theta=q$ and $\lambda=\beta-(n+1+\alpha)/q$, so its discrete condition becomes
\[
    \mu([0,1))<\infty
    \quad\text{and}\quad
    \left\{\frac{\hat{\mu}(r_k)}{\delta_k^{\beta-(n+1+\alpha)/q}}\right\}\in \ell^q,
\]
which is exactly the endpoint necessity established in Proposition \ref{proposition4.8}.}
\end{remark}

% ------------------------------------------------------------------------

\noindent{\bf Acknowledgments}
Tong was supported in part by the National Natural Science Foundation of China (Grant Nos. 12171136 and 12411530045). Yang was supported in part by the National Natural Science Foundation of China (Grant No. 12501103) and the Natural Science Foundation of Hebei Province (Grant No. A2023202031).

\noindent{\bf Declaration}
The authors declare that they have no conflicts of interest. No data were used for the research described in this article.


\begin{thebibliography}{99}

 \bibitem{AguilarHernandezMasPelaezRattya2026}
T. Aguilar-Hern\'andez, A. Mas, J. \'A. Pel\'aez and J. R\"atty\"a,
Maximal theorems for weighted analytic tent and mixed norm spaces,
\emph{J. Funct. Anal.} \textbf{291} (2026), no.\ 3,
Paper No.\ 111513. \href{https://doi.org/10.1016/j.jfa.2026.111513}{doi:10.1016/j.jfa.2026.111513}.

 \bibitem{AlemanSiskakis1997}
A. Aleman and A. G. Siskakis,
Integration operators on Bergman spaces,
\emph{Indiana Univ. Math. J.} \textbf{46} (1997), no. 2, 337--356. \href{https://doi.org/10.1512/iumj.1997.46.1373}{doi:10.1512/iumj.1997.46.1373}.
\bibitem{Andersen1996}
K. F. Andersen,
Ces\`aro averaging operators on Hardy spaces,
\emph{Proc. Roy. Soc. Edinburgh Sect. A} \textbf{126} (1996), 617--624. \href{https://doi.org/10.1017/S0308210500022939}{doi:10.1017/S0308210500022939}.
\bibitem{BHS1965}
A. Brown, P. R. Halmos and A. L. Shields,
Ces\`aro operators,
\emph{Acta Sci. Math. (Szeged)} \textbf{26} (1965), 125--137. \href{https://acta.bibl.u-szeged.hu/14062/}{URL}.
\bibitem{BaoSunWulan2022}
G. Bao, F. Sun and H. Wulan,
Carleson measures and the range of a Ces\`aro-like operator acting on $H^\infty$,
\emph{Anal. Math. Phys.} \textbf{12} (2022), Paper No. 142. \href{https://doi.org/10.1007/s13324-022-00752-z}{doi:10.1007/s13324-022-00752-z}.
\bibitem{BlascoMas2025}
{\'O}. Blasco and A. Mas,
Ces{\`a}ro-type operators on mixed norm spaces,
\emph{Trans. Amer. Math. Soc.}, Advance Publication (2026). \href{https://doi.org/10.1090/tran/9746}{doi:10.1090/tran/9746}.
\bibitem{ChangLiStevic2007}
D.-C. Chang, S. Li and S. Stevi\'c,
On some integral operators on the unit polydisk and the unit ball,
\emph{Taiwanese J. Math.} \textbf{11} (2007), 1251--1285. \href{https://doi.org/10.11650/twjm/1500404862}{doi:10.11650/twjm/1500404862}.
\bibitem{GGM2022}
P. Galanopoulos, D. Girela and N. Merch\'an,
Ces\`aro-like operators acting on spaces of analytic functions,
\emph{Anal. Math. Phys.} \textbf{12} (2022), Paper No. 51. \href{https://doi.org/10.1007/s13324-022-00649-x}{doi:10.1007/s13324-022-00649-x}.
\bibitem{GalGirMer2023}
P. Galanopoulos, D. Girela and N. Merch\'an,
Ces\`aro-type operators associated with Borel measures on the unit disc acting on some Hilbert spaces of analytic functions,
\emph{J. Math. Anal. Appl.} \textbf{526} (2023), Paper No. 127287. \href{https://doi.org/10.1016/j.jmaa.2023.127287}{doi:10.1016/j.jmaa.2023.127287}.
\bibitem{GalMasMer2023}
P. Galanopoulos, D. Girela, A. Mas and N. Merch\'an,
Operators induced by radial measures acting on the Dirichlet space,
\emph{Results Math.} \textbf{78} (2023), Paper No. 106. \href{https://doi.org/10.1007/s00025-023-01887-6}{doi:10.1007/s00025-023-01887-6}.
\bibitem{GSZ2025}
P. Galanopoulos, A. G. Siskakis and R. Zhao,
Weighted Ces\`aro type operators between weighted Bergman spaces,
\emph{Bull. Sci. Math.} \textbf{202} (2025), Paper No. 103622. \href{https://doi.org/10.1016/j.bulsci.2025.103622}{doi:10.1016/j.bulsci.2025.103622}.
\bibitem{Hu2003MixedNorm}
Z. Hu,
Extended Ces\`aro operators on mixed norm spaces,
\emph{Proc. Amer. Math. Soc.} \textbf{131} (2003), no. 7, 2171--2179. \href{https://doi.org/10.1090/S0002-9939-02-06777-1}{doi:10.1090/S0002-9939-02-06777-1}.
\bibitem{Hu2004Bergman}
Z. Hu,
Extended Ces\`aro operators on Bergman spaces,
\emph{J. Math. Anal. Appl.} \textbf{296} (2004), 435--454. \href{https://doi.org/10.1016/j.jmaa.2004.01.045}{doi:10.1016/j.jmaa.2004.01.045}.
\bibitem{JinTang2022}
J. Jin and S. Tang,
Generalized Ces\`aro operators on Dirichlet-type spaces,
\emph{Acta Math. Sci. Ser. B (Engl. Ed.)} \textbf{42} (2022), 212--220. \href{https://doi.org/10.1007/s10473-022-0111-2}{doi:10.1007/s10473-022-0111-2}.
\bibitem{LiStevic2009}
S. Li and S. Stevi{\'c},
Ces{\`a}ro-type operators on some spaces of analytic functions on the unit ball,
\emph{Appl. Math. Comput.} \textbf{208} (2009), no. 2, 378--388. \href{https://doi.org/10.1016/j.amc.2008.12.006}{doi:10.1016/j.amc.2008.12.006}.
\bibitem{Okikiolu1970}
G. O. Okikiolu,
On inequalities for integral operators,
\emph{Glasgow Math. J.} \textbf{11} (1970), 126--133. \href{https://doi.org/10.1017/S0017089500000975}{doi:10.1017/S0017089500000975}.
\bibitem{Siskakis1987}
A. G. Siskakis,
Composition semigroups and the Ces\`aro operator on $H^p$,
\emph{J. London Math. Soc. (2)} \textbf{36} (1987), 153--164. \href{https://doi.org/10.1112/jlms/s2-36.1.153}{doi:10.1112/jlms/s2-36.1.153}.
\bibitem{Siskakis1990}
A. G. Siskakis,
The Ces\`aro operator is bounded on $H^1$,
\emph{Proc. Amer. Math. Soc.} \textbf{110} (1990), 461--462. \href{https://doi.org/10.1090/S0002-9939-1990-1021904-9}{doi:10.1090/S0002-9939-1990-1021904-9}.
\bibitem{Siskakis1996}
A. G. Siskakis,
On the Bergman space norm of the Ces\`aro operator,
\emph{Arch. Math. (Basel)} \textbf{67} (1996), 312--318. \href{https://doi.org/10.1007/BF01197596}{doi:10.1007/BF01197596}.
\bibitem{Stempak1994}
K. Stempak,
Ces\`aro averaging operators,
\emph{Proc. Roy. Soc. Edinburgh Sect. A} \textbf{124} (1994), 121--126. \href{https://doi.org/10.1017/S030821050002922X}{doi:10.1017/S030821050002922X}.
\bibitem{ZhangLiShangGuo2018}
X. Zhang, S. Li, Q. Shang and Y. Guo,
An integral estimate and the equivalent norms on $F(p,q,s,k)$ spaces in the unit ball,
\emph{Acta Math. Sci. Ser. B (Engl. Ed.)} \textbf{38} (2018), 1861--1880. \href{https://doi.org/10.1016/S0252-9602(18)30852-X}{doi:10.1016/S0252-9602(18)30852-X}.
\bibitem{Zhao2015}
R. Zhao,
Generalization of Schur's test and its application to a class of integral operators on the unit ball of $\mathbb C^n$,
\emph{Integral Equations Operator Theory} \textbf{82} (2015), 519--532. \href{https://doi.org/10.1007/s00020-014-2215-0}{doi:10.1007/s00020-014-2215-0}.
\bibitem{ZhaoZhou2022}
R. Zhao and L. Zhou,
$L^p$--$L^q$ boundedness of Forelli--Rudin type operators on the unit ball of $\mathbb C^n$,
\emph{J. Funct. Anal.} \textbf{282} (2022), Paper No. 109345 \ \href{https://doi.org/10.1016/j.jfa.2021.109345}{doi:10.1016/j.jfa.2021.109345}.

\bibitem{ZhaoZhu2008}
R. Zhao and K. Zhu,
\emph{Theory of Bergman Spaces in the Unit Ball of \(\mathbb C^n\)},
M\'emoires de la Soci\'et\'e Math\'ematique de France, no.\ 115,
Soci\'et\'e Math\'ematique de France, Paris, 2008. \href{https://doi.org/10.24033/msmf.427}{doi:10.24033/msmf.427}.
\bibitem{Zhu2005}
K. Zhu,
\emph{Spaces of Holomorphic Functions in the Unit Ball},
Graduate Texts in Mathematics, vol. 226,
Springer-Verlag, New York, 2005. \href{https://doi.org/10.1007/0-387-27539-8}{doi:10.1007/0-387-27539-8}.
\end{thebibliography}
\end{document}